\documentclass[reqno]{amsart}%
\usepackage{graphicx}
\usepackage{amsmath}
\usepackage{amsfonts}
\usepackage{amssymb}%
\usepackage[hidelinks]{hyperref}
\providecommand{\U}[1]{\protect\rule{.1in}{.1in}}
\newtheorem{theorem}{Theorem}

\newtheorem{corollary}[theorem]{Corollary}

\newtheorem{definition}[theorem]{Definition}
\newtheorem{example}[theorem]{Example}

\newtheorem{lemma}[theorem]{Lemma}

\newtheorem{proposition}[theorem]{Proposition}
\newtheorem{remark}[theorem]{Remark}

\newtheorem{assumption}[theorem]{Assumption}
\newtheorem{idea}[theorem]{Idea}
\numberwithin{theorem}{section}
\numberwithin{equation}{section}
\begin{document}
\title[Roots of polynomials under repeated differentiation]{Roots of polynomials under repeated differentiation and repeated applications
of fractional differential operators }
\author{Brian C. Hall}
\address{Brian C. Hall: University of Notre Dame, Notre Dame, IN 46556, USA }
\email{bhall@nd.edu}
\author{Ching-Wei Ho}
\address{Ching-Wei Ho: Institute of Mathematics, Academia Sinica, Taipei 10617, Taiwan }
\email{chwho@gate.sinica.edu.tw}
\author{Jonas Jalowy}
\address{Jonas Jalowy: Institut f\"ur Mathematische Stochastik, Westf\"alische
Wilhelms-Universit\"at\linebreak M\"unster, Orl\'eans-Ring 10, 48149
M\"unster, Germany}
\email{jjalowy@uni-muenster.de}
\author{Zakhar Kabluchko}
\address{Zakhar Kabluchko: Institut f\"ur Mathematische Stochastik, Westf\"alische
Wilhelms-Universit\"at M\"unster, Orl\'eans-Ring 10, 48149 M\"unster, Germany}
\email{zakhar.kabluchko@uni-muenster.de}

\begin{abstract}
We start with a random polynomial $P^{N}(z)$ of degree $N$ with independent
coefficients. We then consider a new polynomial $P_{t}^{N}$ obtained by $\lceil Nt\rceil$ applications
of a fractional 
differential operator of the form $z^{a}
(d/dz)^{b},$ where $a$ and $b$ are real numbers. When $b>0,$ we compute the
limiting root distribution $\mu_{t}$ of $P_{t}^{N}$ as $N\rightarrow\infty.$
We show that $\mu_{t}$ is the push-forward of the limiting root distribution
of $P^{N}$ under a transport map $T_{t}$. The map $T_{t}$ is defined by
flowing along the characteristic curves of a PDE satisfied by the log
potential of $\mu_{t}.$ 

In the special case of repeated differentiation, our
results may be interpreted as saying that the roots evolve radially
\textit{with constant speed} until they hit the origin, at which point, they
cease to exist. For general $a$ and $b,$ the transport map $T_{t}$ has a free
probability interpretation as multiplication of an $R$-diagonal operator by an
$R$-diagonal \textquotedblleft transport operator.\textquotedblright As an application, 
we obtain a push-forward characterization of the free
self-convolution semigroup $\oplus$ of radial measures on $\mathbb{C}$. 

We also consider the case $b<0,$ which includes the case of repeated integration.
More complicated behavior of the roots can occur in this case.

\end{abstract}
\maketitle
\tableofcontents

\section{Introduction}

In this paper, we return to the much-studied question of the evolution of
zeros of high-degree polynomials under repeated differentiation. We also consider the
evolution of zeros under repeated applications of a differential operator of
the form%
\begin{equation}
z^{a}\left(  \frac{d}{dz}\right)  ^{b}. \label{TheOp}%
\end{equation}
For now, the reader may think that $a$ and $b$ are non-negative integers,
although we will eventually allow greater generality. We take the initial
polynomial to be a random polynomial with independent coefficients, of the
sort studied by Kabluchko and Zaporozhets in \cite{KZ}, in which case the
empirical root distribution of the initial polynomial will be asymptotically radial.

We propose that under repeated applications of the operator in (\ref{TheOp}),
the zeros will move approximately under certain \textit{explicit} curves,
depending on the limiting distribution of zeros of the initial polynomial. See
Idea \ref{abMotion.idea} for the precise formula. In the case $a=0$, $b=1$ of
repeated differentiation, our proposal says that the zeros should evolve
radially inward \textit{with constant speed}; see Idea \ref{radialMotion.idea}. We then establish our general
proposal rigorously at the bulk level. This means that we describe the
limiting distribution of zeros of the evolved polynomial as a push-forward of
the limiting distribution of zeros of the original polynomial, under a map
given by the formulas in Idea \ref{abMotion.idea}. Even in the case of
repeated differentiation, this result is new (Theorem \ref{pushRepeated.thm}).

The results of this paper are in the same spirit as in our paper
\cite{HHJK2}, which studies the evolution of zeros of random polynomials
under the heat flow. Both papers present an explicit proposal for how
the zeros move and establish the result rigorously at the bulk level. (See
Theorem 3.3 in \cite{HHJK2} and see also the earlier paper \cite{HeatConj} by
Hall and Ho.) Furthermore, in both cases, the
log potential of the limiting root distribution satisfies a PDE and the
proposed motion of the zeros is along the characteristic curves of the PDE. In
the present paper, the bulk result (Corollary \ref{push.cor}) establishes a
push-forward under a transport map in agreement with Idea \ref{abMotion.idea},
whose trajectories are the characteristic curves of the relevant PDE
(Proposition \ref{charCurves.prop}). Lastly, in Sections
\ref{fracConvolve.sec} and \ref{freeProbInterpret.sec}, we interpret our
results in terms of free probability. (Compare Section 6 in \cite{HHJK2} in
the setting of polynomials undergoing the heat flow.)

\subsection{Prior results on repeated differentiation\label{prior.sec}}

We begin with a basic definition.

\begin{definition}\label{rootmeasure.def}
If $P$ is a nonconstant polynomial in one complex variable, the \textbf{empirical root measure} of
$P$ is the probability measure on $\mathbb{C}$ given by%
\[
\frac{1}{\deg(P)}\sum_{j=1}^{\deg(P)}\delta_{z_{j}},
\]
where $z_{1},\ldots,z_{\deg(P)}$ are the roots of $P,$ listed with their
multiplicity. If $P^{N}$ is a sequence of random polynomials with $\deg(P^N)\rightarrow\infty$, we say that a
(deterministic) probability measure $\mu$ is the \textbf{limiting root
distribution} of $P^{N}$ if the (random) empirical root measure of $P^N$ converges weakly
in probability to $\mu$ as $N\rightarrow\infty$.
\end{definition}

Let $P^{N}$ be such a sequence of random polynomials with $\deg(P^N)=N$ and with limiting root
distribution $\mu.$ The relationship between the zeros of $P^{N}$ and the
zeros of its derivative $dP^{N}/dz$ has been investigated in the physics
literature by Dennis and Hannay \cite{DenHan} and in the mathematics
literature by Pemantle and Rivin~\cite{pemantlerivin},
Subramanian~\cite{subramanian}, Hanin \cite{Han1,Han2}, Kabluchko
\cite{KabluchkoSingleDiff,KabluchkoRepDiff}, O'Rourke~\cite{orourke}, Totik \cite{Totik}, O'Rourke
and Williams~\cite{orourkewilliams}, Kabluchko and Seidel
\cite{KabluchkoSeidel}, Byun, Lee, and Reddy \cite{BLR}, Michelen and Vu
\cite{MichVu}, and Angst, Malicet, and Poly \cite{AMP}. In these works, the
following idea emerges.

\begin{idea}
\label{singleDeriv.idea} Suppose $P^N$ is a sequence of polynomials with limiting root distribution $\mu$ and fix some large value of $N$. Then, upon applying a single derivative, a root $z$ of
$P^{N}$ should move by an amount approximately equal to $1/N$ times the
negative reciprocal of the Cauchy transform of $\mu$ at $z.$
\end{idea}

See, for example, the discussion preceding Conjecture 1.1 in
\cite{KabluchkoRepDiff} or Theorem 2.8 in \cite{orourkewilliams}. For the
readers convenience, we motivate this idea here, following
\cite{KabluchkoRepDiff}.

\begin{proof}
[Heuristic derivation of Idea \ref{singleDeriv.idea}]Denote the zeros of
$P^{N}$ by $z_{1},\ldots,z_{N}$ and let $m(z)$ be the Cauchy transform of the
limiting root distribution $\mu,$ given by%
\begin{equation}
m(z)=\int_{\mathbb{C}}\frac{1}{z-w}~d\mu(w). \label{CauchyDef}%
\end{equation}
We easily compute that%
\begin{equation}
\frac{dP^{N}/dz}{P^{N}(z)}=\sum_{j=1}^{N}\frac{1}{z-z_{j}}. \label{LogDerivP1}%
\end{equation}
If $z$ is close to one of the zeros of $P^{N}$ ---say, the zero $z_{1}$--- the
$j=1$ term on the right-hand side of (\ref{LogDerivP1}) will be larger than
all the others. The remaining terms may be approximated by $N$ times the
Cauchy transform of $\mu$ at $z\approx z_{1}.$ Thus, we expect that%
\begin{equation}
\frac{dP^{N}/dz}{P(z)}\approx\frac{1}{z-z_{1}}+Nm(z_{1}),\quad z\approx z_{1}.
\label{LogDerivP2}%
\end{equation}
Setting the right-hand side of (\ref{LogDerivP2}) equal to zero and solving
for $z$ gives%
\[
z=z_{1}-\frac{1}{Nm(z_{1})}.
\]
This value is the approximate location of a zero of $(dP^{N}/dz)/P^{N}(z)$ and
thus, also, of $dP^{N}/dz.$
\end{proof}

Steinerberger \cite{Stein} then investigated the evolution of polynomials with
real roots under \textit{repeated} differentiation, where the number of
derivatives is proportional to the degree of the polynomial. He introduced a
nonlocal PDE that was conjectured to describe the evolution of the density of
roots. Meanwhile, work of Bercovici and Voiculescu \cite{BV}, further
developed by Nica and Speicher \cite{NicaSpeicher} and Shlyakhtenko and Tao
\cite{STao}, introduced the concept of \textquotedblleft fractional free
convolution,\textquotedblright\ which turned out to describe precisely the
evolution of the density of zeros in Steinerberger's work. Specifically, the
PDE\ in \cite[Eq. (3.6)]{STao} is the same as the one in Steinerberger's work,
up to a scaling. Steinerberger's conjecture was then established rigorously in
work of Hoskins and Kabluchko \cite{HoskinsKabluchko}. A different proof was
given by Arizmendi, Garza-Vargas, and Perales \cite{AGP} using the method of
\textquotedblleft finite free convolution\textquotedblright\ introduced by
Marcus, Spielman, and Srivastava \cite{MSS,Mar}, which then was generalized by Jalowy, Kabluchko and Marynych \cite{JKM1,JKM2} to differential operators including \eqref{TheOp}.

Various authors have then investigated the evolution of zeros under repeated
differentiation when the roots are not real, mainly in the case where the
zeros have an asymptotically radial distribution. Feng and Yao \cite{FengYao}
determined the limiting root distribution for repeated differentiation of
random polynomials with independent coefficients (as in \cite{KZ}). These
polynomials have the property that the limiting root distribution is
rotationally invariant. O'Rourke and Steinerberger \cite{OS} then proposed a
nonlocal PDE for random polynomials whose \textit{roots} (not coefficients)
are i.i.d. with a radial distribution. Hoskins and Kabluchko
\cite{HoskinsKabluchko} then verified that the O'Rourke--Steinerberger PDE
holds in the setting of polynomials with independent coefficients (which, we
note, was not the setting that O'Rourke and Steinerberger considered). Further
work on the evolution of zeros under repeated differentiation was done by
Alazard, Lazar, and Nguyen \cite{ALN}, Kiselev and Tan \cite{KT}, B{\o }gvad,
H{\"{a}}gg and Shapiro \cite{bogvadhaeggshapiro}, Kabluchko
\cite{KabluchkoRepDiff}, Galligo \cite{Galligo}, and Galligo, Najnudel, and Vu \cite{GNV}.

Recent work of Campbell, O'Rourke, and Renfrew \cite{COR} has given an
interpretation of the evolution of zeros in terms of fractional free
convolution for $R$-diagonal operators, analogous to the fractional free
convolution for self-adjoint operators in \cite{BV,NicaSpeicher,STao}. We will
discuss this result further in Section \ref{fracConvolve.sec}.

\subsection{New results on repeated differentiation in the radial
case\label{newResults.sec}}

Although much work has been done on repeated differentiation in the radial
case, one question has remained unanswered, which is to give an explicit
formula for how the zeros move. We propose such a formula here. We note that
since the number of zeros decreases with the number of derivatives, any
description of how the zeros move must include a mechanism for zeros to
\textquotedblleft die\textquotedblright\ at a certain point.

Let $P_{0}^{N}$ be a random polynomial with independent coefficients, as in
\cite{KZ}, and let $\mu_{0}$ be the limiting root distribution of $P_{0}^{N}.$
(Precise assumptions will be stated in Section \ref{PNandQN.sec}.) In that case, $\mu_0$ will be a 
rotationally invariant measure on the plane; see Theorem \ref{KZ.thm} below. Furthermore, essentially all compactly 
supported, rotationally invariant probability measures $\mu_0$ (subject to very mild conditions) occur as the limiting root measure for some 
choice of the random polynomials $P^N_0$.

Although we defer the details of these random polynomials to Section \ref{PNandQN.sec}, we give one explicit example here.

\begin{example}\label{Weyl.example}
The random Weyl polynomials are given by 
\begin{equation}\label{weyldef}
P^N(z)=\sum_{j=0}^N\xi_j\frac{(\sqrt{N}z)^j}{\sqrt{j!}}
\end{equation}
where the $\xi_j$'s are nonconstant i.i.d. random variables with finite log moment. Then
as a special case of a result of Kabluchko and Zaporozhets \cite{KZ} (see Theorem 3.4
and Example 3.3 below), these polynomials have limiting root distribution equal to
the uniform probability measure on the unit disk.
\end{example}

Then let
$P_{t}^{N}$ be the $\lceil Nt\rceil$-th derivative of $P_{0}^{N}$ and let
$\mu_{t}$ be the limiting root distribution of $P_{t}^{N},$ for $0\leq t<1,$
where by Definition \ref{rootmeasure.def}, $\mu_{t}$ is a probability measure. Using results of Feng and Yao
\cite[Theorem 5(2)]{FengYao}, this measure can be computed more-or-less explicitly; 
it is absolutely continuous with respect to the Lebesgue measure and
is rotationally invariant. Let $m_{t}$ be the Cauchy transform of $\mu_{t},$ as in \eqref{CauchyDef}.
Since $\mu_{t}$ is absolutely continuous and rotationally invariant, it is not hard to see
that $zm_{t}(z)$ is always a non-negative real number, namely
\begin{equation}
zm_{t}(z)=\mu_{t}(D_{\left\vert z\right\vert }),\label{zmt}%
\end{equation}
where $D_{r}$ denotes the closed disk of radius $r$ centered at 0.%

\begin{figure}[ptb]%
\centering
\includegraphics[scale=0.3]
{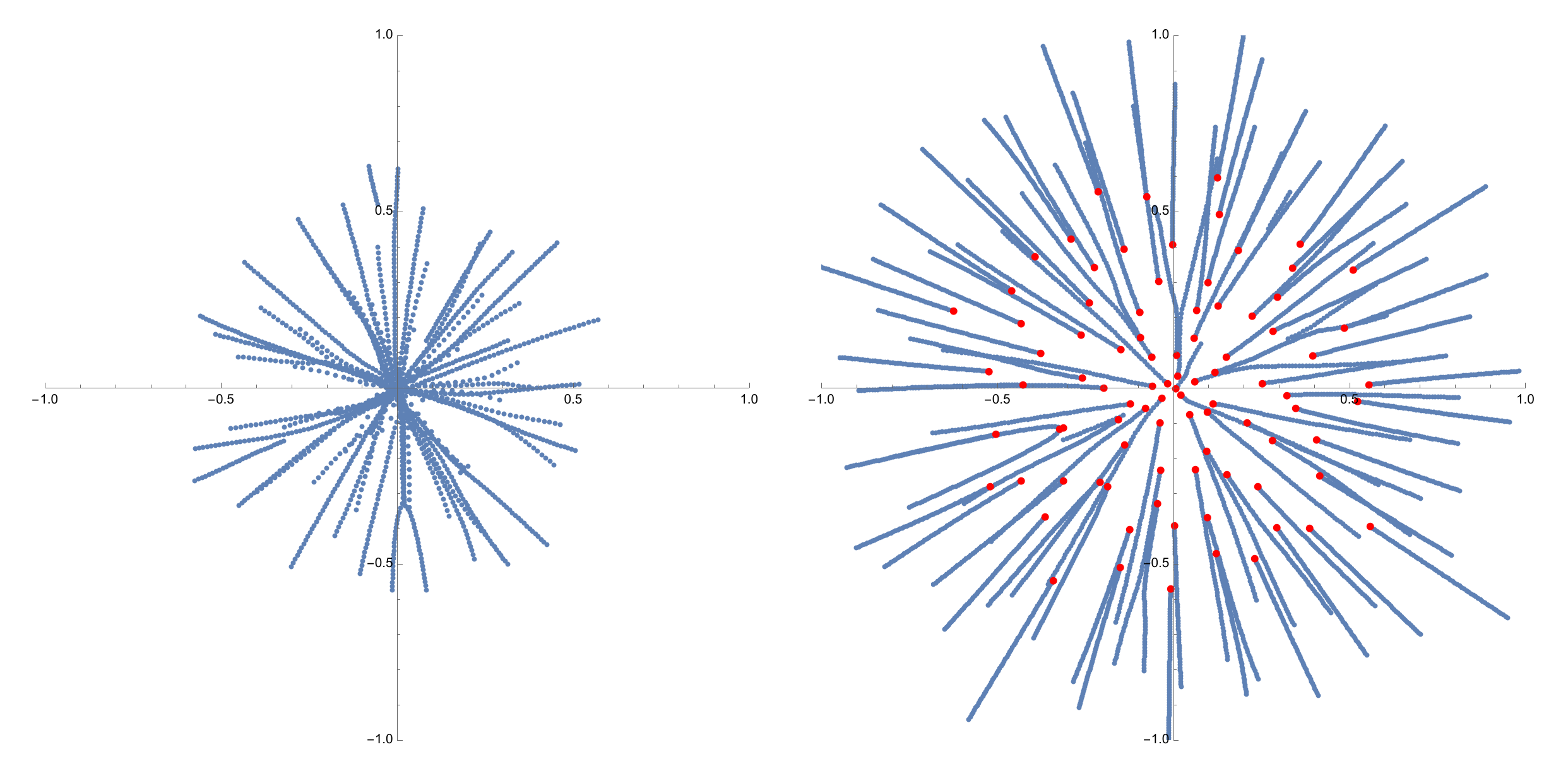}%
\caption{The smaller roots (left) travel radially at constant speed until they
hit the origin and die before time $t.$ The larger roots (right) travel
radially at constant speed without hitting the origin. The blue dots show the
roots of all the polynomials with time $s<t$, while the red dots show the
roots with time $t$. Shown for $t=0.4$ starting from a Weyl polynomial (Example \ref{Weyl.example}) with
$N=300.$}%
\label{bigandsmall.fig}%
\end{figure}

We can then state our proposal for how the zeros move ---and eventually die---
as follows.

\begin{idea}
\label{radialMotion.idea}Let $P_{0}^{N}$ be a random polynomial with
independent coefficients, as in \cite{KZ}, for some fixed, large $N$. Then
consider the roots of the $\lceil Nt\rceil$-th derivative $P_{t}^N$ of $P^N_0$ as a function of $t$, where a time-interval of size
$1/N$ corresponds to a single differentiation. Then each root $z_{0}$ of $P_0^N$ 
moves approximately radially inward \emph{at constant speed equal to }$-1/m_{0}(z_{0})$ 
until it hits the origin, at which point, it ceases to
exist. That is, the roots should approximately follow the curves%
\begin{equation}
z(t)=z_{0}-\frac{t}{m_{0}(z_{0})}=z_{0}\left(  1-\frac{t}{z_{0}m_{0}(z_{0}%
)}\right)  \label{ztGen}%
\end{equation}
for%
\begin{equation}
t<z_{0}m_{0}(z_{0}),\label{tCondition}%
\end{equation}
and the roots should cease to exist when $t\approx z_{0}m_{0}(z_{0}).$

In the case of the Weyl polynomials in Example \ref{LO.example}, $\mu_{0}$ is the uniform
measure on the unit disk and $m_{0}(z)=\bar{z}$ for all $z$ in the unit disk.
In that case, the roots should approximately follow the curves
\[
z(t)=z_{0}-\frac{t}{\bar{z}_{0}}%
\]
for $\left\vert z_{0}\right\vert <1$ and cease to exist when $t\approx
\left\vert z_{0}\right\vert ^{2}.$
\end{idea}

We remark that it is not immediately obvious how Idea \ref{radialMotion.idea} fits with Idea \ref{singleDeriv.idea}. Let $m_t$ denote the 
Cauchy transform of the limiting root distribution of $P_t^N$. If Idea \ref{singleDeriv.idea} holds, the roots should move along
curves $z(t)$ satisfying
\begin{equation}\label{veloc}
\frac{dz}{dt}= - \frac{1}{m_t(z(t))}.
\end{equation} 
Then if Idea \ref{radialMotion.idea} \textit{also} holds, the curves $z(t)$ of the form in \eqref{ztGen} must satisfy \eqref{veloc}. But the curves in \eqref{ztGen} have constant velocity: $dz/dt\equiv -1/m_0(z_0)$. 
Thus, for Ideas \ref{singleDeriv.idea} and \ref{radialMotion.idea} to be consistent, we must have that $m_t(z(t))$ is independent of $t$. That is to say, the Cauchy transform $m_t$, \textit{evaluated along the path of a single root}, must remain constant. 
We give a heuristic derivation of this claim in Appendix \ref{PDEderiv.appendix}.

The condition (\ref{tCondition}) states that a root starting at the point
$z_{0}$ will die before time $t$ precisely if $z_{0}m_{0}(z_{0})<t.$ Thus, the
zeros that die before time $t$ are those with magnitude less than $r,$ where
$r$ is the radius at which $z_{0}m_{0}(z_{0})$ equals $t.$ Thus, by
\eqref{zmt} with $t=0,$ the set of roots that die before time $t$ is
assigned mass $t$ by $\mu_{0},$ meaning that approximately $Nt$ roots die.
This is what we expect when applying $Nt$ derivatives to a polynomial of
degree $N.$

Note also that if $P_{0}^{N}$ is a random polynomial with independent
coefficients, then $P_{t}^{N}$ also has independent coefficients, so that its
distribution of zeros is still asymptotically radial. Thus, the Cauchy
transform of the limiting root distribution will vanish at the origin. Thus,
Idea \ref{singleDeriv.idea} becomes undefined for zeros very close to the
origin. It is therefore plausible that the origin should be the place where
the zeros die as we take repeated derivatives.

We will establish a rigorous version of Idea \ref{radialMotion.idea} at the
bulk level; see Theorem \ref{pushRepeated.thm}.

\begin{remark}
\label{charMotion.rem}One of the key ideas of this paper is that the log
potential $S(z,t)$ of the limiting root distribution of $P_t^N$ satisfies a simple \emph{local} PDE, which
can be solved by the method of characteristics. (See Section \ref{PDE.sec} and
Appendix \ref{PDEderiv.appendix}.) This PDE is to be contrasted
with the nonlocal PDE satisfied by the density of the measure. Idea
\ref{radialMotion.idea} may then be restated in a more fundamental way:
\textbf{The zeros should move approximately along the characteristic curves of
the relevant PDE}, until they reach the origin.
\end{remark}

We actually expect that a variant of Remark \ref{charMotion.rem} should hold in a more general setting, in which
$P_0^N$ is a sequence of (not necessarily random) polynomials of degree $N$ whose empirical root measures are 
converging to a fixed compactly supported probability measure $\mu_0$. The only modification needed to the statement 
is that instead of dying when they reach the origin, the roots will die when they reach a point where the Cauchy transform
of the limiting root distribution of $P_t^N$ is zero. A heuristic argument for this more general statement is given in Appendix 
\ref{PDEderiv.appendix}.

Establishing Idea \ref{radialMotion.idea} rigorously as stated is not easy. We
will, however, prove that the result holds \textit{at the level of the bulk
distribution} of zeros. We let
\[
\alpha_{0}(r)=\mu_{0}(D_{r}),
\]
where $D_{r}$ is the closed disk of radius $r$ centered at the origin. Then we
have the following result.

\begin{theorem}
[Push-forward Theorem for Repeated Differentiation]\label{pushRepeated.thm}Let
$P_{0}^{N}$ be a random polynomial with independent coefficients satisfying
precise assumptions stated in Section \ref{PNandQN.sec}. Let $P_{t}^{N}$ be
the $\lceil Nt\rceil$-th derivative of $P_{0}^{N}$ for $0\leq t<1,$ and let
$\mu_{t}$ be the limiting root distribution of $P_{t}^{N}.$ Assume continuity of the function
\[
r\mapsto\mu_0(D_r),
\]
where $D_r$ is the closed disk of radius $r$ centered at the origin. Let
\[
A_{t}=\left\{  w\in\mathbb{C}:\alpha_{0}(|w|)>t\right\}  .
\]
Define a transport map $T_{t}:A_t\to\mathbb{C}$ by the
right-hand side of (\ref{ztGen}), namely%
\begin{equation}
T_{t}(w)=w-\frac{t}{m_{0}(w)}. \label{TtDef}%
\end{equation}
Then we have the following result connecting $\mu_{t}$ to $\mu_{0}$:%
\[
\mu_{t}=\frac{1}{1-t}(T_{t})_{\#}\left(  \left.  \mu_{0}\right\vert _{A_{t}%
}\right)  ,
\]
where $(T_{t})_{\#}$ denotes push-forward by $T_{t}.$
\end{theorem}

Theorem \ref{pushRepeated.thm} may be interpreted as saying that the limiting
distribution of zeros of $P_{t}^{N}$ \textit{behaves as if} the zeros are
evolving as described in Idea \ref{radialMotion.idea}. More specifically, the
theorem should be interpreted as saying that the small roots of $P_{0}^{N}$
---those with $\left\vert z\right\vert \leq r_{t}$ for $\alpha_{0}(r_{t})=t$
--- die before time $t,$ while the big roots ---those with $\left\vert
z\right\vert >r_{t}$--- evolve according to the curve on the right-hand side
of (\ref{TtDef}).

\subsection{Flowing by fractional derivatives with powers of $z$%
\label{flowFractional.sec}}

We also establish similar results for repeated applications of fractional
differential operators of the form%
\begin{equation}
z^{a}\left(  \frac{d}{dz}\right)  ^{b},\label{abOp}%
\end{equation}
where $a$ and $b$ are real numbers, where the fractional derivative
$(d/dz)^{b}$ is defined on powers of $z$ by a straightforward extension of the
case when $b$ is a positive integer. (See Section \ref{fractional.sec}.) Let us assume for the moment that $a-b$
is rational with denominator $l.$ We then start with a polynomial $P_{0}^{N}$
and apply the operator in (\ref{abOp}) repeatedly, with the following
stipulation: Each time we apply the operator, \textit{we throw away any
negative powers of }$z$\textit{ that arise}. (Negative terms arise only when
$a<b.$) If we then apply (\ref{abOp}) $Nt$ times to a polynomial, assuming
that $Nt$ is an integer multiple of $l,$ we will obtain again a polynomial.
(After applying the operator $l$ times, all powers of $z$ will be integers and
negative powers of $z$ are killed by definition.) The procedure of throwing
away negative terms ensures that the general differential flow behaves
similarly to the case of repeated differentiation. See Remark
\ref{throwaway.rem} for further discussion of this point. When $b>0,$ we will
find a PDE satisfied by the limiting log potential of the polynomials and
establish a push-forward theorem similar to Theorem \ref{pushRepeated.thm}.
The theorem will be the \textquotedblleft bulk\textquotedblright\ version of
the claim that the zeros evolve along the characteristic curves of the
relevant PDE. The characteristic curves, however, will no longer be linear in time.%

\begin{figure}[ptb]%
\centering
\includegraphics[
height=3.6037in,
width=3.6288in
]%
{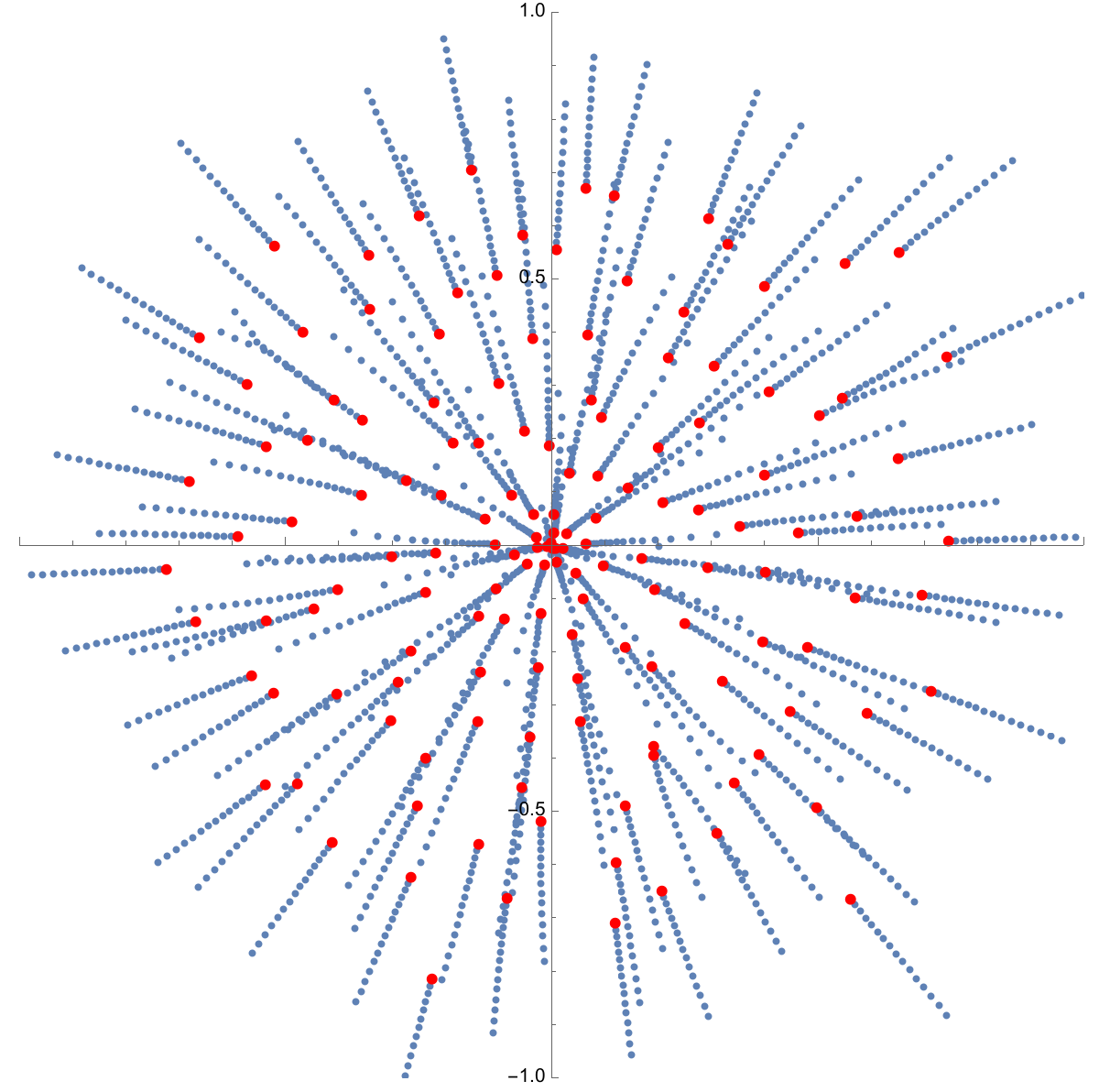}%
\caption{The degree-increasing case with $a=5/2$ and $b=3/2.$ The roots move
radially inward without reaching the origin. Shown for $N=150$ and $t=1/5,$
starting from a Weyl polynomial  (Example \ref{Weyl.example}). The blue dots show the roots at 15 different times less than $t$ and the red dots show the roots at time $t$.} 
\label{degincreasing.fig}%
\end{figure}

In the case that $b>0$ and $a<b,$ the degree of the polynomial decreases with
time and the behavior of the system is similar to the repeated differentiation
case: the zeros will move radially inward and eventually hit the origin. In
the case that $b>0$ and $a>b,$ the degree of the polynomial increases with
time. In that case, the zeros of the original polynomial move radially inward
\textit{without} reaching the origin, while at the same time, zeros are being
created at the origin. See Figure \ref{degincreasing.fig}.

We consider also the case $b<0$ (Section \ref{bLessZero.sec}) which includes
the case of repeated integration ($a=0$ and $b=-1$). This case is more
complicated, in that the limiting root distribution can have mass concentrated
on a circle. This singular behavior arises because the characteristic curves
may collide in this case. See Figure \ref{intweyl.fig}.%

\begin{figure}[ptb]%
\centering
\includegraphics[
height=2.7095in,
width=5.4457in
]%
{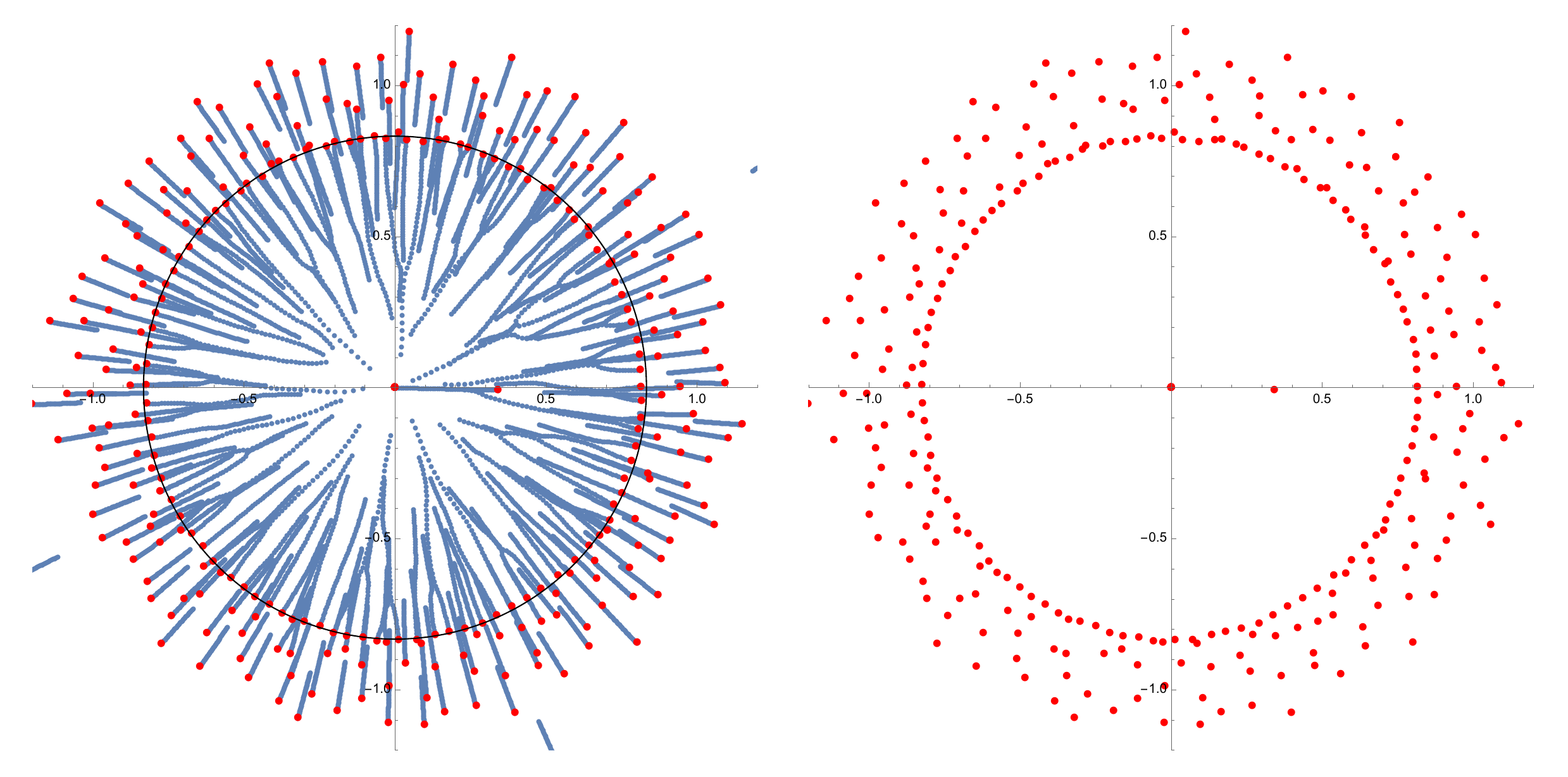}%
\caption{The repeated integration case ($a=0,$ $b=-1$) starting from a Weyl
polynomial (Example \ref{Weyl.example}). The blue dots show the roots of all the polynomials with time
$s<t$, while the red dots show the roots with time $t$. Shown for $t=0.15$ and
$N=300.$}%
\label{intweyl.fig}%
\end{figure}

\section{The fractional flow}

\subsection{Fractional derivatives}\label{fractional.sec}

Recall that the gamma function $\Gamma(z)$ has simple poles at
$z=0,-1,-2,\ldots.$ In what follows, we therefore interpret $1/\Gamma(z)$ as
being zero at these points. When $b$ and $j$ are positive integers, we have%
\begin{align*}
\left(  \frac{d}{dz}\right)  ^{b}z^{j}  &  =j(j-1)\cdots(j-b+1)z^{j-b}\\
&  =\frac{j!}{(j-b)!}z^{j-b}\\
&  =\frac{\Gamma(j+1)}{\Gamma(j+1-b)}z^{j-b},
\end{align*}
where the result is zero (because of the pole in the denominator) when $b>j.$
We then take this formula as a definition for all real numbers $b$ and $j,$
subject to the restriction that the formula should be nonsingular, namely that
$j$ should not be a negative integer. This definition agrees with the
Riemann--Liouville fractional derivative operator with basepoint 0, cf.
\cite{fraccalc}.

\subsection{The general flow\label{generalFlow.sec}}

We now consider applying the operator $z^{a}(d/dz)^{b}$ repeatedly to a
polynomial of degree $N.$ When computing the limiting distribution of roots,
it is convenient to scale this operator by a constant depending on $N,$ which
of course does not affect the roots. Thus, we will actually consider the
operators of the form%
\begin{equation}
\frac{1}{N^{b}}z^{a}\left(  \frac{d}{dz}\right)  ^{b}, \label{firstOp}%
\end{equation}
where $a$ and $b$ are real numbers. The power of $N$ in (\ref{firstOp}) is
chosen so that the coefficients of polynomials obtained by repeated
applications of the operator will have suitable asymptotics as $N\rightarrow
\infty.$

We then apply the operator in (\ref{firstOp}) to powers of $z,$ with the
convention that we \textit{throw away any negative powers of }$z$\textit{ that
arise}, as in the following definition. (See Remark \ref{throwaway.rem} for
the motivation behind this convention.)

\begin{definition}
\label{DAB.def}We define an operator $D^{a,b}$ as follows. For each real
number $j\geq0,$ we define
\begin{equation}
D^{a,b}z^{j}=\frac{1}{N^{b}}\frac{\Gamma(j+1)}{\Gamma(j+1-b)}z^{j+a-b}
\label{dab}%
\end{equation}
when $j+a-b\geq0$ and
\[
D^{a,b}z^{j}=0
\]
when $j+a-b<0.$ Even if $j+a-b\geq0,$ we interpret $D^{a,b}z^{j}$ as being
zero if the gamma function in the denominator on the right-hand side of
(\ref{dab}) is evaluated at a nonpositive integer, i.e., if $b-j$ is a
positive integer. We then extend $D^{a,b}$ linearly to the space of all linear combinations
of powers of $z$.
\end{definition}

If $a-b$ is rational with denominator $l$, then
$(D^{a,b})^{l}$ will map polynomials to polynomials. After all, when $D^{a,b}$
is applied $l$ times to a polynomial, all powers of $z$ will be integers, and
negative powers are thrown away by definition.

\begin{proposition}
Suppose $a-b$ is rational with denominator $l$ and $t>0$ is chosen so that
$Nt$ is an integer multiple of $l.$ Then for a polynomial $P_{0}^{N}$ of
degree $N,$ we define
\[
P_{t}^{N}=(D^{a,b})^{Nt}P_{0}^{N}%
\]
and%
\begin{equation}
Q_{t}^{N}(z)=z^{-Nt(a-b)}P_{t}^{N}(z),\label{QntPnt}%
\end{equation}
where the power of $z$ in the definition of $Q_{t}^{N}$ is chosen so that
$Q_{t}^{N}$ again has degree $N.$ If the coefficient of $z^{j}$ in $P_{0}^{N}$
is $c(N,j),$ then the coefficient $c(N,j,t)$ of $z^{j}$ in $Q_{t}^{N}$ is
computed as follows:%
\begin{equation}
c(N,j,t)=0,\quad0\leq j<Nt(b-a),\label{cnjt0}
\end{equation}
and
\begin{equation}
c(N,j,t)=c(N,j)\cdot\frac{1}{N^{Ntb}}\prod_{0\leq m<Nt}\frac{\Gamma
(j+1+m(a-b))}{\Gamma(j+1+m(a-b)-b)},\label{cnjt}%
\end{equation}
if $j\geq0$ and $Nt(b-a)\leq j\leq N.$
\end{proposition}

Note that in the general situation, we still use the conditions \eqref{cnjt0} and \eqref{cnjt} as written, even though $Nt(b-a)$ need not be an integer. The condition on $j$ in \eqref{cnjt0} is equivalent to the condition
\[
j<\lceil Nt(b-a)\rceil .
\]

\begin{proof}
Direct calculation using (\ref{dab}).
\end{proof}

With the formulas (\ref{cnjt0}) and (\ref{cnjt}) in hand, we see that $Q_{t}^{N}$ 
is well defined and a polynomial for any real numbers $a$ and $b,$ even if
$a-b$ is irrational. We thus make the following general definitions.

\begin{definition}
\label{FlowGeneralAB.def}For any real numbers $a,$ $b,$ and $t$ with $t>0,$ we
define a polynomial $Q_{t}^{N}$ of degree $N$ with coefficients $c(N,j,t)$
defined by (\ref{cnjt0}) and (\ref{cnjt}). Then we define a polynomial
$P_{t}^{N}$ by reversing the roles of $P_{t}^{N}$ and $Q_{t}^{N}$ in
(\ref{QntPnt}):%
\[
P_{t}^{N}(z)=z^{\lceil Nt(a-b)\rceil}Q_{t}^{N}(z).
\]
Here, if $Nt$ is not an integer, the product in (\ref{cnjt}) is understood as
being over all non-negative integers $m$ that are less than $Nt.$ We
generally assume that $t<t_{\max},$ where%
\begin{equation}
t_{\max}=\left\{
\begin{array}
[c]{cc}%
\infty & a\geq b\\
\frac{1}{b-a} & a<b
\end{array}
,\right.  \label{tmax}%
\end{equation}
since $Q_{t}^{N}$ becomes the zero polynomial for $t>t_{\mathrm{\max}}.$
\end{definition}

We now further comment on the convention that when $a<b,$ the coefficients of
$Q_{t}^{N}$ with $j<Nt(b-a)$ are taken to be zero.

\begin{remark}
\label{throwaway.rem}In the case $a<b,$ one \emph{could} define a polynomial
$R_{t}^{N}$ by applying the formula (\ref{cnjt}) for \emph{all} $j\geq0.$ When
$a-b$ is rational, this approach would amount to keeping the negative terms
when iterating the operator $z^{a}(d/dz)^{b},$ and then making all the powers
non-negative at the end by multiplying by $z^{Nt(b-a)}.$ While this approach
gives a well-defined polynomial $R_{t}^{N}$, its roots behave quite
differently from those of $Q_{t}^{N}.$ In particular, the distribution of the
roots of $Q_{t}^{N}$ behaves continuously as $(a,b)$ approaches $(0,1)$ (the
repeated differentiation case), while the distribution of the roots of
$R_{t}^{N}$ behaves discontinuously in this limit.
\end{remark}

To understand Remark \ref{throwaway.rem}, let us consider the case when
$(a,b)$ is close to $(0,1).$ When $a=0$ and $b=1,$ the polynomial $R_{t}^{N}$
is equal to $Q_{t}^{N}$, because repeated differentiation of a polynomial does
not generate any negative powers of $z.$ (That is to say, the expression on
the right-hand side of (\ref{cnjt}) will be zero for $j<Nt(b-a)$ in this
case.) But if we take, say, $a=0$ and $b=0.999,$ and some large value of $N,$
the low-degree coefficients of $Q_{t}^{N}$ will be zero by definition---but
the low-degree coefficients of $R_{t}^{N}$ will not even be small. Figure
\ref{randq.fig} then shows how the roots of $Q_{t}^{N}$ and $R_{t}^{N}$ behave
in this case. The roots of $Q_{t}^{N}$ resemble what we would get for
\emph{either} polynomial when $a=0$ and $b=1,$ namely the roots of the $Nt$-th
derivative of the original polynomial, together with $Nt$ roots at the origin.
By contrast, $R_{t}^{N}$ has a positive fraction of its roots concentrated
near a circle of some positive radius $r.$ (The roots having magnitude greater
than \thinspace$r$ are almost the same for the two polynomials.)%

\begin{figure}[ptb]%
\centering
\includegraphics[
height=2.6187in,
width=5.2278in
]%
{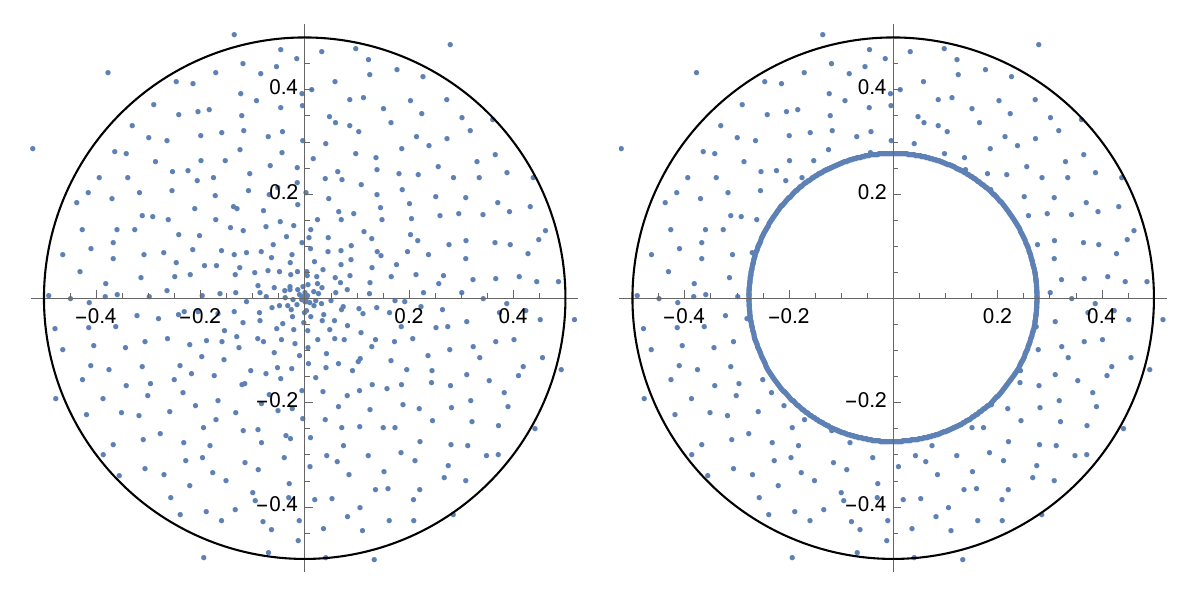}%
\caption{The roots of $Q_{t}^{N}$ (left) and $R_{t}^{N}$ (right) with $a=0$
and $b=0.999,$ starting from a Weyl polynomial. Shown with $N=1000$ and
$t=1/2.$}%
\label{randq.fig}%
\end{figure}

Returning to the situation of general $a$ and $b$, let us now assume that $j\geq0$ and $j\geq Nt(b-a).$ Then the arguments of the
gamma functions in the numerator of the right-hand side of (\ref{cnjt}) are
all at least 1. But for a small number of such $j$'s, the argument of the
gamma function in the \textit{denominator} can be a nonpositive integer, in
which case, we interpret the right-hand side of (\ref{cnjt}) as being zero.
If, for example, $a=3$ and $b=2,$ then $c(N,j,t)=0$ for $j=0$ and~$j=1$,
essentially because $D^{a,b}$ kills $z^{0}$ and $z^{1}.$ (When $j=1$ or $j=1$
and $m=0,$ the gamma function in the denominator is evaluated at a nonpositive
integer.) This sort of behavior causes a small technical difficulty in the
arguments in the next section. The following result will then be useful.

\begin{lemma}
\label{bottom.lem}Define%
\[
j_{\min}=\left\{
\begin{array}
[c]{cc}%
0 & a\geq b\\
\lceil Nt(b-a)\rceil & a<b
\end{array}
\right.  .
\]
Then for each $N$ and each $t$ with $0\leq t<t_{\max},$ there exists a
non-negative integer $j_{0}$ such that for $j\geq0,$ we have $c(N,j_{\min
}+j,t)=0$ if and only if $j<j_{0}.$ Furthermore, the $j_{0}$ depends only on $a$ and $b$; it does not depend on $N$, $t$, $g_0$, or the distribution of the random variables $\xi_j$.
\end{lemma}

What this result means is that problematic coefficients are consecutive, with
indices between $j_{\min}$ and $j_{\min}+j_{0}-1,$ and that the number of
problematic coefficients is bounded.

\begin{proof}
If $c(N,j_{\min}+j,t)=0$ for some $j\geq0,$ then one of the gamma functions in the
denominator of (\ref{cnjt}) must be evaluated at a nonpositive integer when
$j$ replaced by $j_{\min}+j.$ The same will then be true for $c(N,j_{\min
}+k,t)$ for $0\leq k\leq j.$ On the other hand, it is easy to see that if $j$
is big enough---with bounds independent of $N$ and $t$---then all the gamma
functions in the computation of $c(N,j_{\min}+j,t)$ will be evaluated at
positive numbers. We thus take $j_{0}$ to be one more than the largest
$j\geq0$ for which $c(N,j_{\min}+j,t)=0,$ if such a $j$ exists, and we take
$j_{0}$ to be zero, otherwise.
\end{proof}

\section{The exponential profiles of $P_{t}^{N}$ and $Q_{t}^{N}$%
\label{PNandQN.sec}}

In this section, we consider a class of random polynomials with independent
coefficients, studied by Kabluchko and Zaporozhets in \cite{KZ}. We show that
if $P_{0}^{N}$ is of this type, so are the polynomials $P_{t}^{N}$ and
$Q_{t}^{N}$ in Definition \ref{FlowGeneralAB.def}. In the case of repeated
differentiation ($a=0$ and $b=1$), this result was obtained previously by Feng
and Yao \cite[Theorem 5(2)]{FengYao}, using different scaling conventions. 

\subsection{Random polynomials with independent coefficients}

Consider i.i.d. random variables $\xi_{0},\xi_{1},\ldots$ that are
nondegenerate (i.e., not almost surely constant) and satisfy%
\[
\mathbb{E}\log(1+\left\vert \xi_{j}\right\vert )<\infty.
\]
Then, for each $N,$ consider a random polynomial of degree $N$ of the form
\begin{equation}
P_{0}^{N}(z)=\sum_{j=0}^{N}\xi_{j}a_{j;N}z^{j}\label{PNform}%
\end{equation}
for certain deterministic complex constants $\{a_{j;N}\}.$ Thus, the coefficients of
$P_{0}^{N}$ are independent (but typically not identically distributed).

We now assume that the constants $a_{j;N}$ have a particular asymptotic
behavior as $N$ tends to infinity. We consider a function
\[
g:[0,1]\rightarrow\lbrack-\infty,\infty),
\]
where we emphasize that $-\infty$ is an allowed value. For our purposes, it is
convenient to make the following assumptions on $g.$

\begin{assumption}
\label{gProperties.assumption}There is some $\alpha_{\min}$ with $0\leq
\alpha_{\min}<1$ such that the following conditions hold:

\begin{enumerate}
\item $g(\alpha)=-\infty$ for $0\leq\alpha<\alpha_{\min}$

\item $g$ is finite, continuous, and concave on $(\alpha_{\min},1],$ so that
$g$ has a left derivative, denoted by $g^{\prime}$, at each point in
$(\alpha_{\min},1]$. In addition, $g$ is right-continuous at $\alpha_{\min}$.

\item $g^{\prime}(1)>-\infty.$
\end{enumerate}
\end{assumption}

We emphasize that $g$ need not be \textit{strictly} concave on $(\alpha_{\min},1]$. We also note that $g(\alpha_{\min})$ is allowed to have the value $-\infty$; see Remark \ref{KZtechnicalities.remark}. 

We then make the following assumption on the deterministic coefficients
$a_{j;N}$.

\begin{assumption}
\label{KZcoeff.assumption}The coefficients $a_{j;N}$ satisfy
\begin{equation*}
a_{j;N}=0\quad\text{if }j/N<\alpha_{\min}
\end{equation*}
and
\begin{equation}
\vert a_{j;N}\vert=e^{Ng(j/N)+o(N)} \label{aFromG}%
\end{equation}
in the following precise sense:
\begin{equation}
\lim_{N\rightarrow\infty}\sup_{0\leq j\leq N}\left\vert \left\vert
a_{j;N}\right\vert ^{1/N}-e^{g(j/N)}\right\vert =0, \label{gDefNew}%
\end{equation}
where $e^{g(j/N)}$ is interpreted be zero when $g(j/N)=-\infty.$
\end{assumption}
The function $g$ is called the \textbf{exponential profile} of the random
polynomial $P_{0}^{N}.$

Suppose that $g(\alpha_{\min})$ is finite, so that $g$ is bounded on $[\alpha_{\min},1]$. Then by the uniform
continuity of the exponential and logarithm functions on closed intervals, the following
condition is equivalent to (\ref{gDefNew}):%
\begin{equation}
\lim_{N\rightarrow\infty}\sup_{\alpha_{\min}N\leq j\leq N}\left\vert \frac
{1}{N}\log\left\vert a_{j;N}\right\vert -g(j/N)\right\vert =0.\label{gDef}%
\end{equation}
In the $a=b$ case of Theorem \ref{QNprofile.thm}, we will consider a profile $g_t$ with $\alpha_{\min}=0$ and $g_t(0)=-\infty$, in which case, 
we really need to use \eqref{gDefNew} rather than \eqref{gDef}.

It is possible to remove the assumption that $g$ be concave on $[\alpha_{\min
},1]$. But then the limiting behavior of the zeros of $P_{0}^{N}$ is
determined by the concave majorant of $g,$ that is, the smallest concave
function that is greater than or equal to $g.$ We will restrict our attention
to the concave case until Section \ref{bLessZero.sec}.

\begin{example}
[Littlewood--Offord polynomials]\label{LO.example}For all $\beta\geq0,$ the
function%
\begin{equation}
g(\alpha)=-\beta(\alpha\log\alpha-\alpha) \label{LOprofile}%
\end{equation}
defines a concave exponential profile with $\alpha_{\min}=0.$ The associated
random polynomials are called Littlewood--Offord polynomials with parameter
$\beta.$ One concrete realization of the Littlewood--Offord polynomials is given by 
\begin{equation*}
P_0^N(z)=\sum_{j=0}^N \xi_j\frac{(N^{\beta}z)^j}{(j!)^{\beta}},
\end{equation*}
which is (up to a scaling of the variable) the form introduced by Littlewood and Offord in \cite[Theorem 3]{LO}. Special cases include $\beta=0$ (Kac polynomials), $\beta=1/2$ (Weyl polynomials as in \eqref{weyldef}), and $\beta=1$ (exponential polynomials).
\end{example}

Note that when $\beta=0$ in (\ref{LOprofile}), we have $g\equiv0$ and we may
take the deterministic coefficients $a_{j;N}$ to be equal 1 for all $j$ and
$N.$ In that case, the coefficients of the associated polynomial are
independent and identically distributed, as in the work of Littlewood--Offord \cite{LO} and Kac \cite{Kac}.

We now state a main result of \cite{KZ}, in a level of generality convenient
for our applications.

\begin{theorem}
[Kabluchko--Zaporozhets]\label{KZ.thm}Let $P_{0}^{N}$ be a random polynomial
with independent coefficients, with a fixed exponential profile $g$ satisfying
Assumptions \ref{gProperties.assumption} and \ref{KZcoeff.assumption}. Then
the empirical root measure of $P_{0}^{N}$ converges weakly in probability to a
rotationally invariant probability measure $\mu.$ The radial part of $\mu$ is
the push-forward of the uniform measure on $[0,1]$ under the map $r(\alpha)$
given by%
\begin{equation}
r(\alpha)=\left\{
\begin{array}
[c]{cc}%
0 & 0\leq\alpha\leq\alpha_{\min}\\
e^{-g^{\prime}(\alpha)} & \alpha>\alpha_{\min}%
\end{array}
\right.  .\label{RofAlpha}%
\end{equation}
That is to say, $\mu$ is the unique rotationally invariant probability measure
on $\mathbb{C}$ such that the distribution of $\,r=\left\vert z\right\vert $
is the above push-forward measure. In particular, $\mu$ will consist of
$\alpha_{\min}$ times a $\delta$-measure at the origin, plus a measure
supported on an annulus with inner and outer radii given by%
\begin{equation}
r_{\mathrm{in}}=\lim_{\alpha\rightarrow\alpha_{\min}^{+}}e^{-g^{\prime}%
(\alpha)},\quad r_{\mathrm{out}}=e^{-g^{\prime}(1)}.\label{InnerOuter}%
\end{equation}
By Assumption \ref{gProperties.assumption}, $r_{\mathrm{out}}$ is finite, and
the inner radius $r_{\mathrm{in}}$ can be positive or zero.
\end{theorem}

If we consider, for example, the Weyl polynomials (taking $\beta=1/2$ in
(\ref{LOprofile})), we have $g^{\prime}(\alpha)=-\frac{1}{2}\log\alpha$ and%
\[
r(\alpha)=e^{-g^{\prime}(\alpha)}=\sqrt{\alpha}.
\]
This means that $r^{2}=\alpha$ is uniformly distributed on $[0,1],$ indicating
that $\mu$ is the uniform probability density on the unit disk.

\begin{remark}\label{KZtechnicalities.remark}
Theorem \ref{KZ.thm} corresponds to the case of \cite[Theorem 2.8]{KZ} where the constant $T_0$ in Assumption (A1) of \cite{KZ} equals 1. Theorem 2.8 of \cite{KZ} actually assumes $\alpha_{\min}=0$, but the general case of Theorem \ref{KZ.thm} can easily be reduced to the case 
$\alpha_{\min}=0$, by pulling out a factor of $z^k$, where $k$ is the largest integer less than $\alpha_{\min} N$. See Section \ref{PandQ.sec} where in the case $a<b$, the polynomials $P_t^N$ have $\alpha_{\min}>0$, while the associated polynomials $Q_t^N$ (with a large power of $z$ factored out) have $\alpha_{\min}=0$. 

Assuming now that $\alpha_{\min}=0$, we note that Theorem 2.8 in \cite{KZ} assumes (in Assumption (A1)) that $g(\alpha_{\min})>-\infty$---that is, in the notation of \cite{KZ}, that $f(0)>0$. This assumption, however, is not actually needed in the proof. Specifically, the proof works with the function $f(\alpha):=e^{g(\alpha)}$. As long as this function is continuous at $0$, no change in the proof is needed if $f(0)=0$.

We note that even if the initial polynomial $P_0^N$ satisfies the exact assumptions of \cite[Theorem 2.8]{KZ}, the exponential profiles of the polynomials $Q_t^N$ and $P_t^N$ may have a nonzero value of $\alpha_{\min}$ and (in the case $a=b$) may have $g(\alpha_{\min})=-\infty$. See Theorem \ref{QNprofile.thm} and Proposition \ref{PNprofile.prop}. 
\end{remark}

\subsection{The case of $P_{t}^{N}$ and $Q_{t}^{N}$}\label{PandQ.sec}

In this section, we will derive the exponential profiles for the polynomials
$P_{t}^{N}$ and $Q_{t}^{N},$ when the initial polynomial has independent
coefficients, extending the results of Feng--Yao \cite[Theorem 5(2)]{FengYao}
for the case of repeated differentiation. Specifically, we make the following
assumption on the initial polynomial $P_{0}^{N}.$

\begin{assumption}
\label{InitialPoly.assumption}The initial polynomial $P_{0}^{N}$ in the
definition of $P_{t}^{N}$ and $Q_{t}^{N}$ is a polynomial with independent
coefficients satisfying Assumptions \ref{gProperties.assumption} and
\ref{KZcoeff.assumption} and having $\alpha_{\min}=0.$ We denote the
exponential profile of $P_{0}^{N}$ by $g_{0}.$
\end{assumption}

Before coming to the main result of this section, we must address a small
technical issue. The issue is that the coefficient $c(N,j,t)$ of $Q_{t}^{N}$
can be zero even if $j\geq Nt(b-a),$ as discussed at the end of Section
\ref{generalFlow.sec}. Because of this, $Q_{t}^{N}$ may not satisfy the
estimates (\ref{gDefNew}). But if $j_{0}$ is as in Lemma \ref{bottom.lem}, we
can simply pull out a factor of $z^{j_{0}}$ from $Q_{t}^{N}$ and the new
polynomial $\tilde Q^N_t$ of degree $N-j_{0}$ will no longer have this problem.

Now, Lemma \ref{bottom.lem} says that $j_0$ is bounded independent of $N$. Thus, $\tilde Q^N_t$ has the same 
limiting root distribution as $Q_{t}^{N}$ itself. Furthermore, since the function $\alpha\mapsto e^g(\alpha)$ is 
continuous by Assumption \ref{gProperties.assumption}, replacing $N$ by $N-j_0$ in \eqref{gDefNew} does not affect the limit.

Throughout the rest of the paper, we assume that $P_{0}^{N}$ is a random
polynomial with independent coefficients with exponential profile $g_{0}$
satisfying Assumption \ref{InitialPoly.assumption}. In some cases, we will
impose additional assumptions on $g_{0}.$ We then introduce the polynomials
$P_{t}^{N}$ and $Q_{t}^{N}$ in Definition \ref{FlowGeneralAB.def}. The
polynomial $Q_{t}^{N}$ will have the same form as in (\ref{PNform}), except
that the deterministic constants will have changed. The new deterministic
constants $a_{j;N}^{t}$ can be read off from (\ref{cnjt}) as
\begin{equation}
a_{j;N}^{t}=a_{j;N}\cdot\frac{1}{N^{Ntb}}\prod_{m=0}^{Nt-1}\frac
{\Gamma(j+1+m(a-b))}{\Gamma(j+1+m(a-b)-b)}\label{ajnt}%
\end{equation}
for non-negative integers $j$ satisfying $j\geq Nt(b-a),$ with $a_{j;N}^{t}$
being zero if $j<Nt(b-a).$

\begin{theorem}
[Exponential profile of $Q_{t}^{N}$]\label{QNprofile.thm}Assume that the
random polynomial $P_{0}^{N}$ satisfies Assumption
\ref{InitialPoly.assumption}. When $a=b$ we further assume that $b>0.$ Assume
$t<t_{\max},$ where $t_{\max}$ is as in (\ref{tmax}). Define $\alpha_{\min
}^{t}$ as%
\begin{equation}
\alpha_{\min}^{t}=\left\{
\begin{array}
[c]{cc}%
0 & a\geq b\\
t(b-a) & a<b
\end{array}
\right.  .\label{alpha0}%
\end{equation}
When $a\neq b,$ define%
\begin{equation}
g_{t}(\alpha)=g_{0}(\alpha)+\frac{b}{a-b}\{[\alpha+t(a-b)]\log[\alpha
+t(a-b)]-\alpha\log\alpha\}-bt\label{gtAlpha}%
\end{equation}
for $\alpha\geq\alpha_{\min}^{t}$ and $g_{t}(\alpha)=-\infty$ for
$\alpha<\alpha_{\min}^{t}.$ In that case, we have%
\begin{equation}
(g_{t}(\alpha)-g_{0}^{{}}(\alpha))^{\prime}=\frac{b}{a-b}\{\log[\alpha
+t(a-b)]-\log\alpha\}\label{gtAlphaD1}%
\end{equation}
and%
\begin{equation}
(g_{t}(\alpha)-g_{0}(\alpha))^{\prime\prime}=-\frac{b}{a-b}\left(  \frac
{1}{\alpha}-\frac{1}{\alpha+t(a-b)}\right)  ,\label{gtAlphaD2}%
\end{equation}
for $\alpha>\alpha_{\min}^{t}.$ When $a=b$ (and $b>0$), the formulas are
obtained by letting $a$ approach $b$ in (\ref{gtAlpha})--(\ref{gtAlphaD2}):%
\begin{align}
g_{t}(\alpha) &  =g_{0}(\alpha)+bt\log\alpha\label{gtEqual}\\
(g_{t}(\alpha)-g_{0}(\alpha))^{\prime} &  =\frac{bt}{\alpha}\label{gtEqualD1}%
\\
(g_{t}(\alpha)-g_{0}(\alpha))^{\prime\prime} &  =-\frac{bt}{\alpha^{2}%
}.\label{gtEqualD2}%
\end{align}
Then $g_{t}$ is the exponential profile of $Q_{t}^{N},$ in the sense that the
coefficients $a_{j;N}^{t}$ of the polynomial $Q_{t}^{N}$ satisfy the error
estimate in (\ref{gDefNew}) with respect to the function $g_{t},$ after
removing a factor of $z^{j_{0}}$ as in Lemma \ref{bottom.lem}. Thus, if
$g_{t}$ is concave on $[\alpha_{\min}^{t},1],$ it will satisfy
Assumptions \ref{gProperties.assumption} and \ref{KZcoeff.assumption}.
Concavity will hold when $b\geq0$ and may also hold when $b<0,$ depending on
the choice of $g_{0}.$
\end{theorem}

The proof will be given later in this subsection. Note that the right-hand
sides of (\ref{gtAlphaD2}) and (\ref{gtEqualD2}) are negative when $b>0,$
provided that $\alpha$ is greater than the quantity $\alpha_{\min}^{t}$ in
(\ref{alpha0}). Thus, when $b>0$, we see that $g_t$ is the sum of two concave functions and 
is therefore concave. When $b<0,$ the right-hand sides of (\ref{gtAlphaD2}) and
(\ref{gtEqualD2}) are positive for $\alpha>\alpha_{\min}^{t},$ in which case,
$g_{t}$ will only be concave if the concavity of $g_{0}$ outweighs the failure
of concavity of $g_{t}-g_{0}.$ See Section \ref{bLessZero.sec}. The reason for
the restriction $b>0$ in the case $a=b$ is that if $b<0,$ then in
(\ref{gtEqual}), we would have $g_{t}(0)=+\infty$, which is not allowed in
Theorem \ref{KZ.thm} (because it violates Assumption \ref{gProperties.assumption}).

\begin{remark}
Assume $b>0$. Since $-g_0'$ is already increasing, it follows from \eqref{gtEqualD2} that the map $\alpha\mapsto -g_t'(\alpha)$ is \textit{strongly} increasing, in the sense
that there is a constant $c$ with $-g_t(\alpha_1)+g_t(\alpha_2)>c(\alpha_1-\alpha_2)$ whenever $\alpha_1>\alpha_2$. Theorem \ref{KZ.thm} then implies 
that the limiting root distribution of $Q_t^N$ is the sum of a $\delta$-measure at the origin and a measure that is absolutely continuous with respect to the two-dimensional Lebesgue measure.
\end{remark}

In the case of repeated differentiation ($a=0$ and $b=1$), the expression in
(\ref{gtAlpha}) is equivalent to the expression in Theorem 5(2) in the paper
\cite{FengYao} of Feng and Yao, after accounting for minor differences of
normalization and notation. (The main difference is in how one accounts for the change
in degree introduced by differentiation.)

We note that if $a<b,$ it is convenient to rewrite the formulas in a way that
makes it more obvious which quantities are positive. Thus, for example, we
have%
\[
g_{t}(\alpha)=g_{0}(\alpha)+\frac{b}{b-a}\{\alpha\log\alpha-[\alpha
-t(b-a)]\log[\alpha-t(b-a)]\}-bt.
\]

We also note a natural scaling property of the results in Theorem
\ref{QNprofile.thm}.

\begin{remark}
\label{scaling.remark} All formulas in Theorem \ref{QNprofile.thm} are
unchanged if we multiply $a$ and $b$ by a positive constant $c$ and then
divide $t$ by $c$.
\end{remark}

\begin{corollary}
\label{rt.cor}If we use the notation $(\cdot)^{\prime}$ for the left
derivative, then for $\alpha>\alpha_{\min}^{t},$ we have%
\begin{equation}
e^{-g_{t}^{\prime}(\alpha)}=e^{-g_{0}^{\prime}(\alpha)}\left(  \frac
{\alpha+t(a-b)}{\alpha}\right)  ^{\frac{b}{b-a}},\quad a\neq b
\label{expGprime1}%
\end{equation}
and%
\begin{equation}
e^{-g_{t}^{\prime}(\alpha)}=e^{-g_{0}^{\prime}(\alpha)}e^{-\frac{bt}{\alpha}%
},\quad a=b. \label{expGprime2}%
\end{equation}

\end{corollary}

The functions on the right-hand sides of (\ref{expGprime1}) and
(\ref{expGprime2}) are closely connected to the characteristic curves of the
PDE we will consider in Section \ref{PDE.sec}.

Before turning to the proof of Theorem \ref{QNprofile.thm}, we record also the
exponential profile for $P_{t}^{N}$ and give two examples.

\begin{proposition}[Exponential profile of $P^N_t$]
\label{PNprofile.prop}Under the assumptions of Theorem \ref{QNprofile.thm},
the exponential profile $h_{t}$ of the polynomial $P_{t}^{N}$ of degree
$N+\lceil Nt(a-b)\rceil$ is given by%
\begin{equation}
h_{t}(\alpha)=\frac{1}{(1+t(a-b))}g_{t}\left(  \alpha(1+t(a-b))-t(a-b)\right)
\label{htAlpha}%
\end{equation}
for $\alpha>\beta_{\min}^{t},$ where 
\begin{equation*}
\beta_{\min}^{t}=\left\{
\begin{array}
[c]{cc}%
0 & a\leq b\\
\frac{t(a-b)}{1+t(a-b)} & a>b
\end{array}.
\right. 
\end{equation*}
The function $h_{t}$ equals $-\infty$ for $\alpha<\beta_{\min}^{t}.$ As in Theorem \ref{QNprofile.thm}, we interpret the the result as saying that the
coefficients of the polynomial $P_{t}^{N}$ satisfy the error
estimate in (\ref{gDefNew}) with respect to the function $g_{t},$ after
removing a factor of $z^{j_{0}}$ as in Lemma \ref{bottom.lem}.
\end{proposition}

\begin{proof}
Since $P_{t}^{N}(z)=z^{\lceil Nt(a-b)\rceil}Q_{t}^{N}(z)$, the exponential profile $h_{t}$
of $P_{t}^{N}$ can be obtained as a rescaling of the exponential
profile $g_{t}$ of $Q_{t}^{N}.$
\end{proof}

\begin{example}
[Stability of the Littlewood--Offord distribution]\label{LOstable.example}%
Suppose $g_{0}$ is the exponential profile of the Littlewood--Offord
polynomials with parameter $\beta>0,$ namely%
\[
g_{0}(\alpha)=-\beta(\alpha\log\alpha-\alpha).
\]
Then suppose we choose $a$ and $b$ with $b>0$ so that
\[
\frac{b}{b-a}=\beta.
\]
Then (\ref{htAlpha}) takes the form%
\begin{equation}
h_{t}(\alpha)=-\beta(\alpha\log\alpha-\alpha)-\alpha\log[(1+t(a-b))^{\beta}].
\label{htStable}%
\end{equation}
In this case, $a<b,$ so $\beta_{\min}^{t}=0$ and the limiting root
distribution of $P_{t}^{N}$ is again the limiting root distribution of the
Littlewood--Offord polynomials with parameter $\beta$, dilated by a factor of
$(1+t(a-b))^{\beta}.$

If $a$ and $b$ are such that $b/(b-a)=\beta>0$ but $b<0,$ then $a$ must be greater than $b$, so that $\beta^t_{\min}$ is positive. In that case, 
the formula
(\ref{htStable}) still applies for $\alpha>\beta_{\min}^{t}$. In this case, the limiting root distribution of $P_{t}^{N}$ will
agree with a dilation of the limiting distribution of the Littlewood--Offord
polynomials, but only outside a disk of radius
\[
r_{\mathrm{in}}(t)=((a-b)t)^{\beta}.
\]
Note that in this case, the exponential profile $h_{t}$ is concave on
$(\beta_{\min}^{t},1],$ even if $b<0.$ See Section \ref{IntegrateExp.sec}.
\end{example}

\begin{proof}
The formula (\ref{htStable}) is obtained from (\ref{htAlpha}) by simplifying.
Meanwhile, subtracting a term of the form $\alpha\log c$ from the exponential
profile is easily seen to have the same effect as rescaling the variable in
the associated polynomial by $1/c$ (i.e., $p(z)\mapsto p(z/c)$), which
multiplies all the zeros by $c.$
\end{proof}

\begin{example}
[Kac case]Assume $g_{0}\equiv0$ (Kac case) and~$b>0.$ Then for $t<t_{\max},$
we have $r_{\mathrm{in}}(t)=0$ and
\[
r_{\mathrm{out}}(t)=\left\{
\begin{array}
[c]{cc}%
(1+(a-b)t)^{\frac{b}{b-a}} & a\neq b\\
e^{-bt} & a=b
\end{array}
.\right.
\]
The limiting root distribution $\mu_t$ of the polynomials $P^N_t$ is given explicitly in three cases as
\[
d\mu_{t}=\mathbf{1}_{\{r<r_{\mathrm{out}}\}}\frac{1}{1+t(a-b)}\frac
{t(a-b)^{2}}{b}\frac{r^{a/b}}{(r-r^{a/b})^{2}}~dr~\frac{d\theta}{2\pi},\quad
a<b,
\]
and%
\[
d\mu_{t}=\mathbf{1}_{\{r<r_{\mathrm{out}}\}}\frac{bt}{r\log(r)^{2}}%
~dr~\frac{d\theta}{2\pi},\quad a=b,
\]
and%
\[
d\mu_{t}=t(a-b)\delta_{0}+\mathbf{1}_{\{r<r_{\mathrm{out}}\}}(1-t(a-b))\frac
{t(a-b)^{2}}{b}\frac{r^{a/b}}{(r-r^{a/b})^{2}}~dr~\frac{d\theta}{2\pi},\quad
a>b.
\]

In the repeated differentiation case ($a=0,$ $b=1$), we get%
\[
r_{\mathrm{out}}(t)=(1-t)
\]
and%
\[
d\mu_{t}=\mathbf{1}_{\{r<1-t\}}\frac{1}{1-t}\frac{t}{(1-r)^{2}}dr~\frac
{d\theta}{2\pi}.
\]

\end{example}

\begin{proof}
We first compute with the polynomial $Q_{t}^{N}$, whose exponential profile is
$g_{t}.$ With $g_{0}\equiv0,$ we compute that for $a\neq b,$ we have%
\begin{equation}
r_{t}(\alpha):=e^{-g_{t}^{\prime}(\alpha)}=\left(  \frac{\alpha+t(a-b)}%
{\alpha}\right)  ^{\frac{b}{b-a}},\quad\alpha\in\lbrack\alpha_{\min}^{t},1].
\label{rtKac}%
\end{equation}
The formulas for $r_{\mathrm{in}}$ and $r_{\mathrm{out}}$ follow by evaluating
at $\alpha=\alpha_{\min}^{t}$ and at $\alpha=1.$ We can then solve the formula
for $r_{t}(\alpha)$ for $\alpha$ as
\[
\alpha_{t}(r)=\frac{t(b-a)}{1-r^{\frac{b-a}{b}}}.
\]
Then the limiting root distribution $\sigma_{t}$ of $Q_{t}^{N}$ has mass
$\alpha_{\min}^{t}$ at the origin and mass $\alpha_{t}(r)$ in the disk of
radius $r,$ for $0<r<r_{\mathrm{out}}(t).$ Then
\[
\alpha_{t}^{\prime}(r)=\frac{t(a-b)^{2}}{b}\frac{r^{a/b}}{(r-r^{a/b})^{2}}%
\]
gives the density of the distribution of $r$ for $r>0.$ It is then
straightforward to convert these results into results for the limiting root
distribution $\mu_{t}$ of $P_{t}^{N}.$ The case $a=b$ can be analyzed
similarly by letting $a\rightarrow b$ in (\ref{rtKac}).
\end{proof}

We prepare for the proof of Theorem \ref{QNprofile.thm}, with the following
well-known asymptotic of the Gamma function, which follows directly from
Stirling's formula; see for instance \cite[5.11.12]{Nist}.

\begin{lemma}
\label{lem:GammaOverGamma} For each fixed $b\in\mathbb{R},$ we have
\begin{equation}
\log\frac{\Gamma(z)}{\Gamma(z-b)}=b\log(z)+o(1) \label{GammaOverGamma}%
\end{equation}
as $z\rightarrow\infty.$
\end{lemma}



\begin{proof}
[Proof of Theorem \ref{QNprofile.thm}]Let us define $G_{t}(\alpha)$ so that
\[
g_{t}(\alpha)=g_{0}(\alpha)+G_{t}(\alpha),
\]
where the formula for $G_{t}(\alpha)$ can be read off from \eqref{gtAlpha} and \eqref{gtEqual}. We
then define $G_{t}^{N}$ at numbers of the form $j/N,$ by
\begin{equation}
G_{t}^{N}(j/N)=\frac{1}{N}\log\left(  \frac{1}{N^{Ntb}}\prod_{m=0}^{Nt-1}%
\frac{\Gamma(j+1+m(a-b))}{\Gamma(j+1+m(a-b)-b)}\right)  ,\label{GtN}%
\end{equation}
so that, by \eqref{ajnt}, we have%
\[
\left\vert a_{j;N}^{t}\right\vert ^{1/N}=\left\vert a_{j;N}\right\vert
^{1/N}e^{G_{t}^{N}(j/N)}.
\]
Then we compute that%
\begin{align}
\left\vert a_{j;N}^{t}\right\vert ^{1/N}-e^{g_{t}(j/N)}  & =\left(  \left\vert
a_{j;N}\right\vert ^{1/N}-e^{g_{0}(j/N)}\right)  e^{G_{t}^{N}(j/N)}%
\nonumber\\
& +e^{g_{0}(j/N)}(e^{G_{t}^{N}(j/N)}-e^{G_{t}(j/N)}).\label{aMinusEg}%
\end{align}
Provided that $G_{t}^{N}(j/N)$ is bounded above, uniformly in $j$ and $N,$ the
first term on the right-hand side of (\ref{aMinusEg}) will tend to zero
uniformly in $j$ as $N\rightarrow\infty,$ by Assumption \ref{KZcoeff.assumption}. Since, also,
$e^{g_{0}(\alpha)}$ is bounded, it will then suffice to show that
$e^{G_{t}^{N}(j/N)}$ converges to $e^{G_{t}(j/N)}$ uniformly in $j$ as
$N\rightarrow\infty.$

\textbf{The case }$a>b.$ For the case $a>b,$ we will show that $G_{t}%
^{N}(j/N)$ converges uniformly to $G_{t}(j/N)$ as $N\rightarrow\infty,$ where
$G_{t}^{N}$ is as in (\ref{GtN}) and from \eqref{gtAlpha},
\[
G_{t}(\alpha)=\frac{b}{a-b}\{[\alpha+t(a-b)]\log[\alpha+t(a-b)]-\alpha
\log\alpha\}-bt.
\]
This uniform convergence will also give a uniform upper bound on $G_{t}%
^{N}(j/N).$ 

Define%
\begin{equation}
w=j+1+m(a-b)-b,\label{ww}%
\end{equation}
which is the quantity appearing in the gamma functions in the denominator in
(\ref{GtN}). As discussed at the beginning of this subsection, there will be
some $j_{0}\geq0$ such that for $j<j_{0},$ the value of $w$ is a nonpositive
integer for some $m$---causing $a_{j;N}^{t}$ to be zero---while for $j\geq
j_{0},$ this does not occur. Even if $j\geq j_{0},$ there may be some values
of $m$ for which $w$ is negative or very close to zero. We therefore choose
$m_{0}^{j}$ so that
\begin{equation}
j+1+m_{0}^{j}(a-b)-b\geq a-b.\label{m0def}%
\end{equation}
(This particular lower bound will be convenient below.) When $j\geq a-1,$ we
may take $m_{0}^{j}=0.$ Thus, there are only finitely many $j$'s for which
$m_{0}^{j}$ is not zero. For these $j$'s, the finitely many values of $m$ with
$m<m_{0}^{j}$ will not affect the large-$N$ asymptotics of the expression in
(\ref{GtN}).

We then rewrite (\ref{GtN}) in a way suggested by applying
(\ref{GammaOverGamma}) with
\begin{equation}
z=j+1+m(a-b),\label{zz}%
\end{equation}
namely%
\begin{align}
G_{t}^{N}(j/N) &  \approx\frac{1}{N}\sum_{m=m_{0}^{j}}^{Nt-1}\left(
\log\left(  \frac{\Gamma(j+1+m(a-b))}{\Gamma(j+1+m(a-b)-b)}\right)
-b\log(j+1+m(a-b))\right)  \label{second}\\
&  +\frac{1}{N}\sum_{m=m_{0}^{j}}^{Nt-1}\left(  b\log(j+1+m(a-b))-b\log
N\right)  ,\label{third}%
\end{align}
where the symbol $\approx$ indicates that we are dropping the terms with
$m<m_{0}^{j}.$

Now, the sum in (\ref{second}) equals $t$ times the average of the first approximately $Nt$
terms of the sequence%
\begin{equation}
c_{m}^{j}=\log\left(  \frac{\Gamma(j+1+m(a-b))}{\Gamma(j+1+m(a-b)-b)}\right)
-b\log(j+1+m(a-b)), \label{cmj}%
\end{equation}
which tends to zero as $m\rightarrow\infty,$ by applying (\ref{GammaOverGamma}%
) with $z=j+1+m(a-b)$.

Moreover, if $m_{0}^{j}$ is as in (\ref{m0def}), then the quantity $z=w+b$ in
(\ref{zz}) is at least $a,$ so that $z-b=w$ is at least $a-b>0.$ Then
using (\ref{GammaOverGamma}), we can see that the function
\[
z\mapsto\log\frac{\Gamma(z)}{\Gamma(z-b)}-b\log(z)
\]
is bounded for $z$ in the interval $[\max(a,1),\infty).$ Thus, the sum in
(\ref{second}) tends to zero as $N$ tends to infinity, uniformly in $j.$

We then write the sum in (\ref{third}) using the notation $\alpha_{j}=j/N$ as
\begin{equation}
\frac{b}{N}\sum_{m=m_{0}^{j}}^{Nt-1}\log\left(  \alpha_{j}+\frac{1}{N}%
+\frac{m}{N}(a-b)\right)  .\label{RiemannSum}%
\end{equation}
This is a Riemann sum approximation to the quantity%
\[
b\int_{0}^{t}\log[\alpha_{j}+x(a-b)]~dx=G_{t}(\alpha_{j}).
\]
By (\ref{m0def}), the argument of the logarithm in (\ref{RiemannSum}) is
always at least $a-b,$ which is the lattice spacing in the Riemann sum. Then
since the log function is increasing, we can estimate the sum from above and
below by integrals, as in the integral test for convergence of sums. It is
then straightforward to see that we get convergence of (\ref{third}) to
$G_{t}(\alpha_{j}),$ uniformly in $j.$ Since we have already shown that
(\ref{second}) tends to zero uniformly in $j,$ we conclude that $G_{t}%
^{N}(j/N)$ converges uniformly to $G_{t}(j/N)$. 

\textbf{The case }$a<b.$ This case is extremely similar to the case $a>b,$
except that we only consider $j\geq Nt(b-a).$ We again use Lemma
\ref{bottom.lem}, which tells us in this case that there will be some
$j_{0}\geq0$ such that $a_{j;N}^{t}$ is zero for $j<Nt(b-a)+j_{0}$ but nonzero otherwise.

\textbf{The case }$a=b,$ with $b>0.$ In this case, from \eqref{gtEqual},
\begin{equation}
G_{t}(\alpha)=bt\log\alpha,\label{GtEqual}%
\end{equation}
which approaches $-\infty$ as $\alpha$ approaches 0. In this case, we will not
get uniform convergence of $G_{t}^{N}(j/N)$ to $G_{t}(j/N)$ but will still get
uniform convergence of $e^{G_{t}^{N}(j/N)}$ to $e^{G_{t}(j/N)}.$ This
convergence guarantees a uniform upper bound on $G_{t}^{N}(j/N).$ We will
divide the analysis into two cases, the case in which $j/N$ is small, in which
case both $e^{G_{t}^{N}(j/N)}$ and $e^{G_{t}(j/N)}$ will be close to zero, and
the case in which $j/N$ is not small, in which case, we can argue similarly to
the case $a>b.$

When $a=b,$ all the factors in the product on the right-hand side of
(\ref{GtN}) are equal and we obtain%
\begin{equation}
e^{G_{t}^N(j/N)}=\left(  \frac{1}{N^{b}}\frac{\Gamma(j+1)}{\Gamma(j+1-b)}%
\right)  ^{t}\label{ajFormula}%
\end{equation}
and this quantity tends to zero as $N\rightarrow\infty,$ for each fixed $j.$

We now pick some $\delta\in(0,1)$ and divide our analysis into two cases:
$j\leq\delta N$ and $j>\delta N.$ In the first case, we note that the digamma
function $\psi(z)=\Gamma^{\prime}(z)/\Gamma(z)$ is increasing for $z>0,$ as a
consequence of a standard integral representation. We can then easily verify
that $\Gamma(j+1)/\Gamma(j+1-b)$ is increasing for $j>b-1.$ Thus, for $j>b-1,$
we have%
\[
\frac{1}{N^{b}}\frac{\Gamma(j+1)}{\Gamma(j+1-b)}\leq\frac{1}{N^{b}}%
\frac{\Gamma(\delta N+1)}{\Gamma(\delta N+1-b)}.
\]
But by (\ref{GammaOverGamma}) with $z=\delta N+1,$ we have%
\[
\lim_{N\rightarrow\infty}\frac{1}{N^{b}}\frac{\Gamma(\delta N+1)}%
{\Gamma(\delta N+1-b)}=\delta^{b}.
\]
Thus, for any $\delta,$ if $N$ is large enough, the conditions $j\leq\delta N$ and $j>b-1$ give
\[
\frac{1}{N^{b}}\frac{\Gamma(j+1)}{\Gamma(j+1-b)}\leq2\delta^{b}.
\]

Thus, given any $\varepsilon>0,$ if we choose $\delta$ small enough, then for
all sufficiently large $N,$ we will have
\[
e^{G_{t}^{N}(j/N)}\leq\varepsilon/2;\quad e^{G_{t}(j/N)}\leq\varepsilon/2,
\]
for all $j/N\leq\delta.$ (The finitely many $j$'s with $j\leq b-1$ cause no
problem, since the right-hand side of (\ref{ajFormula}) tends to zero as
$N\rightarrow\infty$ with $j$ fixed.) With $\delta$ chosen in this way, we now
consider the case $j>\delta N.$ We compute from (\ref{GtN}) and (\ref{GtEqual}%
) that
\[
G_{t}^{N}(j/N)-G_{t}(j/N)=t\left(  \log\left(  \frac{\Gamma(j+1)}%
{\Gamma(j+1-b)}\right)  -b\log(j)\right)  .
\]
This quantity tends to zero as $N\rightarrow\infty,$ uniformly in $j>\delta
N,$ as a consequence of (\ref{GammaOverGamma}). 
\end{proof}

\section{The push-forward theorem\label{PushForward.sec}}

We now explore how the roots of $P_{t}^{N}$ or $Q_{t}^{N}$ move as $t$ varies,
generalizing Idea \ref{radialMotion.idea} in the case of repeated
differentiation. The results of this section are parallel to results of
\cite[Section 3]{HHJK2}, where we investigated how the zeros of polynomials
evolve under the heat flow.

In this section, we assume that the exponential profile $g_{t}$ of $Q_{t}%
^{N},$ as computed in Theorem \ref{QNprofile.thm}, is concave, so that Theorem
\ref{KZ.thm} applies. In that case, the exponential profile $h_{t}$ in
Proposition \ref{PNprofile.prop} will also be concave. By Theorem
\ref{QNprofile.thm}, concavity of $g_{t}$ will hold if $b\geq0$ (assuming, of
course, concavity of $g_{0}$). As we will see in Section \ref{bLessZero.sec},
concavity of $g_{t}$ may also hold when $b<0,$ depending on the choice of
$g_{0}.$ For the desired results to make sense, we need to assume
\textit{strict} concavity of $g_{0}$ (Remark \ref{pushAssumptions.remark}). We
summarize these assumptions as follows.

\begin{assumption}
\label{r0.assumption}The function $g_{0}$ is \emph{strictly} concave on
$(0,1]$ and the function $g_{t}$ is concave on the interval $(\alpha_{\min
}^{t},1].$
\end{assumption}

We let $\mu_{0}$ be the limiting root distribution of $P_{0}^{N}$ and we
define a function $\alpha_{0}:[0,\infty)$ by
\begin{equation}
\alpha_{0}(r)=\mu_{0}(D_{r}),\label{alpha0def}%
\end{equation}
where $D_{r}$ is the closed disk of radius $r$ centered at 0. The assumption
that $g_{0}$ is strictly concave guarantees that the left-derivative
$g_{0}^{\prime}$ cannot be constant on any interval, so that $\mu_{0}$ does
not give mass to any circle. Thus, $\alpha_{0}$ is continuous---but not
necessarily strictly increasing.

If $\mu_{0}$ is sufficiently regular, its Cauchy transform $m_{0},$ defined by%
\begin{equation}
m_{0}(z)=\int_{\mathbb{C}}\frac{1}{z-w}~d\mu_{0}(w), \label{m0integral}%
\end{equation}
exists as an absolutely convergent integral for every $z.$ By an elementary
calculation, $m_{0}$ satisfies
\begin{equation}
m_{0}(z)=\frac{\alpha_{0}(\left\vert z\right\vert )}{z},\quad z\neq0.
\label{m0formula}%
\end{equation}
For general $\mu_{0}$, assuming only Assumption \ref{r0.assumption}, we simply
take (\ref{m0formula}) as the definition of $m_{0}$. Note that%
\[
zm_{0}(z)=\alpha_{0}(\left\vert z\right\vert )\geq0.
\]

\begin{idea}
\label{abMotion.idea}The zeros of $P_{t}^{N}$ should move approximately along
curves suggested by the right-hand side of Corollary \ref{rt.cor}, where
$\alpha$ in the corollary is identified with $zm_{0}(z)$ at the starting point
of the curve. That is to say, we take%
\begin{equation}
z(t)=z_{0}\left(  \frac{z_{0}m_{0}(z_{0})+t(a-b)}{z_{0}m_{0}(z_{0})}\right)
^{\frac{b}{b-a}},\quad a\neq b \label{zt1}%
\end{equation}
and%
\begin{equation}
z(t)=z_{0}\exp\left\{  -\frac{bt}{z_{0}m_{0}(z_{0})}\right\}  ,\quad a=b.
\label{zt2}%
\end{equation}
When $a\geq b,$ this motion should hold for all $t>0,$ but if $a>b$, zeros are
also being created at the origin. When $a<b,$ the motion holds until the curve
hits the origin, at which point, the zero ceases to exist.
\end{idea}

It is possible to motivate this idea by a PDE argument, generalizing the
argument in Section \ref{newResults.sec} for the case $a=0$ and~$b=1,$ but
omit the details. We only remark that the curves $z(t)$ in Idea
\ref{abMotion.idea} are the \textbf{characteristic curves} of the PDE
satisfied by the log potential of the limiting root distribution of $P_{t}%
^{N}.$ (See Proposition \ref{charCurves.prop} in Section \ref{hj.sec}).

We will establish Idea \ref{abMotion.idea} rigorously at the bulk level in
Theorem \ref{push.thm} and Corollary \ref{push.cor} below. We now sketch the
argument for this bulk result. We let $\mu_{t}$ and $\sigma_{t}$ denote the
limiting root distributions of $P_{t}^{N}$ and $Q_{t}^{N},$ respectively. In
particular, $\mu_{0}=\sigma_{0}$ is the limiting root distribution of
$P_{0}^{N}=Q_{0}^{N}.$

Let us consider at first the degree-nondecreasing case, $a\geq b$,
where~$\alpha_{\min}^{t}=0.$ Define for $t\geq0,$ a map $r_{t}%
:[0,1]\rightarrow\lbrack r_{\mathrm{in}}(t),r_{\mathrm{out}}(t)]$, where
$r_{\mathrm{in}}(t)$ and $r_{\mathrm{out}}(t)$ are defined by (\ref{InnerOuter}), by
\[
r_{t}(\alpha)=e^{-g_{t}^{\prime}(\alpha)}.
\]
If $r_{t}$ is continuous and strictly increasing on $[0,1],$ it has a
continuous inverse---and by Theorem \ref{KZ.thm}, that inverse will be the
function $\alpha_{0}$ in (\ref{alpha0def}). Under Assumption
\ref{r0.assumption}, the function $\alpha_{0}$ is continuous and it is then
easy to see that the distribution of $\alpha_{0}$ with respect to $\mu_{0}$ is
uniform on $[0,1].$ That is, pushing forward the radial part of $\mu_{0}$ by
$\alpha_{0}$ gives the uniform measure on $[0,1].$

Meanwhile, according to Theorem \ref{KZ.thm}, the radial part of the limiting
root distribution $\sigma_{t}$ of $Q_{t}^{N}$ is the push-forward of the
uniform measure on $[0,1]$ under $r_{t}.$ Note that we start with the
\textit{same} uniform measure on $[0,1]$ for all $t.$ We therefore have a
push-forward result: The map
\[
r_{t}\circ\alpha_{0}
\]
will take the radial part of $\mu_{0}$ to the uniform measure on $[0,1]$ and
then to the radial part of $\mu_{t}.$ Meanwhile, Corollary \ref{rt.cor} tells
us that for $a>b,$ we have%
\[
(r_{t}\circ\alpha_{0})(r)=\left(  \frac{\alpha_{0}(r)+t(a-b)}{\alpha_{0}%
(r)}\right)  ^{\frac{b}{b-a}}r,
\]
with a limiting formula for $a=b.$

The preceding discussion leads to the following definition of a transport map
in the case $a\geq b,$ with a natural modification in the case $a<b.$

\begin{definition}
\label{Tt.def}When $a\geq b,$ we define a transport map $T_{t}$ as follows:%
\[
T_{t}(re^{i\theta})=(r_{t}\circ\alpha_{0})(r)\cdot e^{i\theta}.
\]
Explicitly, we have
\begin{equation}
T_{t}(w)=%
\begin{cases}
w\left(  \frac{\alpha_{0}(|w|)+t(a-b)}{\alpha_{0}(|w|)}\right)  ^{\frac
{b}{b-a}}\quad & a>b\\
we^{-\frac{bt}{\alpha_{0}(|w|)}} & a=b,
\end{cases}
\label{Tt1}%
\end{equation}
where $\alpha_{0}$ is as in (\ref{alpha0def}). When $a<b$, we define $T_{t}$
by
\begin{equation}
T_{t}(w)=%
\begin{cases}
w\left(  \frac{\alpha_{0}(|w|)-t(b-a))}{\alpha_{0}(|w|)}\right)  ^{\frac
{b}{b-a}}\quad & \alpha_{0}(|w|)\geq t(b-a)\\
0 & \text{otherwise}%
\end{cases}
. \label{Tt2}%
\end{equation}

\end{definition}

We note that the formulas on the right-hand side of (\ref{Tt1}) and
(\ref{Tt2}) agree with those in (\ref{zt1}) and (\ref{zt2}), if we
identify---as in (\ref{m0formula})---the quantity $z_{0}m_{0}(z_{0})$ in
(\ref{zt1}) and (\ref{zt2}) with the quantity $\alpha_{0}(\left\vert
w\right\vert )$ in (\ref{Tt1}) and (\ref{Tt2}).

The main result is the following. Compare Theorem 3.3 in \cite{HHJK2} for
polynomials evolving under the heat equation.

\begin{theorem}[Push-forward result for $Q^N_t$]
\label{push.thm}Suppose Assumption \ref{r0.assumption} holds, along with our
usual Assumption \ref{InitialPoly.assumption}. Suppose also that the
exponential profile $g_{t}$ of $Q_{t}^{N}$ is concave, which will hold, for
example, if $b>0.$ Finally, if $a=b,$ assume $b>0.$ Let $\sigma_{t}$ be the
limiting root distribution of the polynomial $Q_{t}^{N}$ in Definition
\ref{FlowGeneralAB.def}. Then%
\begin{equation}
\sigma_{t}=(T_{t})_{\#}(\mu_{0}), \label{Sigmat}%
\end{equation}
where $(T_{t})_{\#}$ denotes push-forward by $T_{t}.$
\end{theorem}

\begin{remark}
\label{pushAssumptions.remark}Theorem \ref{push.thm} cannot hold as stated
without Assumption \ref{r0.assumption}. If, for example, the initial
polynomials $P_{0}^{N}$ are the Kac polynomials (corresponding to the case
$g_{0}\equiv0$), then $\mu_{0}$ will be the uniform measure on the unit
circle. But the limiting root distribution $\sigma_{t}$ of $Q_{t}^{N}$ will be
absolutely continuous with respect to Lebesgue measure on the plane, for
$b>0,$ since the exponential profile $g_{t}$ will be strictly concave when
$g_{0}\equiv0.$ In that case, $\sigma_{t}$ cannot be the push-foward of
$\mu_{0}$ under any rotationally invariant map.
\end{remark}

In the Kac case, one can work around this difficulty by \textquotedblleft
randomizing\textquotedblright\ the measure $\mu_{0}.$ This amounts to
attaching a random variable $\alpha\in\lbrack0,1]$ to each point in the unit
circle. Then we let $\tilde{\mu}_{0}$ be the joint distribution of $w$ and
$\alpha,$ where $w$ is uniform on the unit circle and $\alpha$ is uniform on
$[0,1],$ independent of $w.$ (That is, we imagine that even if $\mu_{0}$
concentrates onto the unit circle, the quantity $\alpha:=\alpha_{0}(\left\vert
w\right\vert )$ in (\ref{alpha0def}) is still uniformly distributed between
$0$ and $1.$) Then $\sigma_{t}$ will be the push-forward of $\tilde{\mu}_{0}$
under a modified transport map $\tilde{T}_{t}(w,\alpha)$, where $\alpha
_{0}(\left\vert w\right\vert )$ in the definition of $T_{t}(w)$ is replaced by
$\alpha.$ Thus, for $a>b,$ we would have%
\[
\tilde{T}_{t}(w,\alpha)=w\left(  \frac{\alpha+t(a-b)}{\alpha}\right)
^{\frac{b}{b-a}},\quad a>b.
\]
Indeed, the push-forward property of $\tilde{T}_{t}$ is nothing but Theorem
\ref{KZ.thm} for the polynomials $Q_{t}^{N}$, using Corollary \ref{rt.cor}.
One can do something similar in general by randomizing on each circle that is
assigned positive mass by $\mu_{0},$ but we omit the details of this construction.

We then restate Theorem \ref{push.thm} in terms of $P_{t}^{N},$ with separate statements for the case
$a\geq b$ and $a<b$.

\begin{corollary}[Push-forward result for $P^N_t$]
\label{push.cor}Continue with the hypotheses of Theorem \ref{push.thm} and let
$\mu_{t}$ be the limiting root distribution of $P_{t}^{N}.$ Then when $a\geq
b,$ we have
\begin{equation}
\mu_{t}=\frac{1}{1+t(a-b)}((T_{t})_{\#}(\mu_{0})+t(a-b)\delta_{0}).
\label{mut1}%
\end{equation}
When $a<b,$ let $A_{t}$ be the set of $w$ with $\alpha_{0}(\left\vert
w\right\vert )\geq t(b-a).$ Then we have
\begin{equation}
\mu_{t}=\frac{1}{1-t(b-a)}(T_{t})_{\#}(\left.  \mu_{0}\right\vert _{A_{t}}).
\label{mut2}%
\end{equation}

\end{corollary}

\begin{proof}
In the case $a\geq b,$ the limiting root measure $\mu_{t}$ of $P_{t}^{N}$ is
obtained from the limiting root measure $\sigma_{t}$ of $Q_{t}^{N}$ by adding
a multiple of a $\delta$-measure at the origin and then rescaling the
resulting measure to be probability measure, so that (\ref{mut1}) follows from
(\ref{Sigmat}). In the case $a<b,$ suppose we push forward the restriction of
$\mu_{0}$ to $A_{t}$ by the map $T_{t}$ in Definition \ref{Tt.def}. Then since
$T_{t}$ maps the complement of $A_{t}$ to 0, we get the measure $\sigma_{t},$
with a multiple of a $\delta$-measure at the origin removed. But the result is
then just the measure $\mu_{t},$ up to scaling by a constant.
\end{proof}

We now prove Theorem \ref{push.thm}.

\begin{proof}
[Proof of Theorem \ref{push.thm}]Since the measures on both sides of
(\ref{Sigmat}) are rotationally invariant, it suffices to check that they have
the same radial part (i.e., the same distribution of the radius). But this
claim follows from the discussion prior to the statement of the theorem.
Push-forward by $\alpha_{0}$ takes the radial part of $\mu_{0}$ to the uniform
measure on $[0,1]$ and then push-forward by $r_{t}$ take the uniform measure
on $[0,1]$ to the radial part of $\sigma_{t},$ by Theorem \ref{KZ.thm}. The
composite map $r_{t}\circ\alpha_{0}$ is the computed by Corollary \ref{rt.cor}
and agrees with $T_{t}.$ (In the case $a<b,$ both $r_{t}\circ\alpha_{0}(z)$
and $T_{t}(z)$ give the value 0 when $\alpha_{0}(\left\vert w\right\vert
)<\alpha_{\min}^{t}=t(b-a).$)
\end{proof}

\section{Free probability interpretation: additive self-convolution\label{fracConvolve.sec}}

In the paper \cite{COR}, Campbell, O'Rourke, and Renfrew establish a
connection between the repeated differentiation flow in the radial case and an
operation that the authors call fractional free convolution for rotationally
invariant probability measures. This work thus gives the first free
probability interpretation to the repeated differentiation flow, connecting it
to the notion of sums of freely independent $R$-diagonal operators. In this
section, we note that the fractional convolution in \cite{COR} also has a
close connection to the differential flow analyzed in the present paper, in
the case $a=-1,$ $b=1.$ Their results should be compared to work on repeated
differentiation of polynomials with all real roots, as discussed in the
introduction, just after the heuristic derivation of Idea
\ref{singleDeriv.idea}. We refer to the monographs of Nica and Speicher \cite{NicaSpeicherLecture} and 
Mingo and Speicher \cite{MingoSpeicher} for basic information about free probability and $R$-diagonal 
operators.

The authors of \cite{COR} define an operation $(\cdot)^{\oplus k},$ acting on
rotationally invariant probability measures on the plane. Here $k$ is a real
number with $k\geq1.$ The case in which $k$ is an integer has been studied
previously by Haagerup and Larsen \cite{HL} and by K\"{o}sters and Tikhomirov
\cite{KostersTikhomirov}. This operation may be connected to the theory of
$R$-diagonal operators in free probability in two different ways. Let $A$ be an
$R$-diagonal operator with (rotationally invariant) Brown measure $\mu.$ First,
suppose that $k\geq1$ is an integer. Then $\mu^{\oplus k}$ is the Brown
measure of $A_{1}+\cdots+ A_{k},$ where $A_{1},\ldots,A_{k}$ are
freely independent copies of $A.$ Second, suppose that $k\geq1$ is any real
number and let $q$ be a projection freely independent of $A$ with the trace of
$q$ equal to $1/k.$ Then $\mu^{\oplus k}$ is the Brown measure of $kqAq,$
where $qAq$ is viewed as an element of the compressed von Neumann algebra
associated to $q.$ (See Section 4.1 of \cite{COR}.)

We mention two basic examples. First, if $\mu$ is the uniform probability
measure on the unit disk, $\mu^{\oplus k}$ is the uniform probability measure
on the disk of radius $\sqrt{k}.$ Second, if $\mu$ is the uniform probability
measure on the unit circle, $\mu^{\oplus k}$ describes the limiting eigenvalue
distribution of truncations (i.e., corners) of Haar-distributed unitary
matrices. (See the work of \.{Z}yczkowski and Sommers \cite{ZS} in the physics
literature and Petz and R\'{e}ffy \cite{PetzReffy} in the math literature.)
See Sections 5.1.1 and 5.1.2 in \cite{COR}.

Now, Definition 4.1 in \cite{COR} defines $\mu^{\oplus k}$ under the
assumption that $\mu$ is the Brown measure of an $R$-diagonal element. But
then in Eq. (4.7), the authors write a relation between the quantile functions
of $\mu$ and $\mu^{\oplus k}$ that makes sense for arbitrary radial
probability measure. We therefore adopt \cite[Eq. (4.7)]{COR} as the
definition of $\mu^{\oplus k}$ in general.

For our purposes, it is convenient to make a minor rescaling of the flow in
\cite{COR}. We define $(\cdot)^{\hat{\oplus}k}$ so that if $\mu$ is the Brown
measure of $A,$ then $\mu^{\hat{\oplus}k}$ is the Brown measure of $qAq$
(rather than $kqAq$), where, as above, the trace of $q$ equals $1/k.$ Then
$\mu^{\hat{\oplus}k}$ is simply the push-forward of $\mu^{\oplus k}$ under
the map consisting of \textquotedblleft multiplication by $1/k$%
.\textquotedblright

Now suppose that $P_{t}^{N}$ is a polynomial with independent coefficients
satisfying Assumption \ref{InitialPoly.assumption}, undergoing repeated
differentiation, that is, the differential flow with $a=0$ and $b=1.$ Then
Campbell, O'Rourke, and Renfrew apply a squaring operation, denoted $\psi_{2}$
in \cite{COR}, which amounts to considering the polynomial $\hat{P}_{t}^{N}$
given by
\[
\hat{P}_{t}^{N}(z)=P_{t}^{N}(z^{2}).
\]
That is, $\hat{P}_{t}^{N}$ is the polynomial whose zeros are the
\textit{square roots of the zeros of }$P_{t}^{N}$ (taking always \textit{both}
square roots of each zero of $P_{t}^{N}$).

\begin{theorem}
[Campbell, O'Rourke, Renfrew]Let $P_{0}^{N}$ be a sequence of polynomials with
independent coefficients satisfying Assumption \ref{InitialPoly.assumption}
and let $P_{t}^{N}$ be obtained from $P_{0}^{N}$ by repeated differentiation.
Let $\nu_{t}$ denote the limiting root distribution of the polynomial%
\begin{equation}
z\mapsto P_{t}^{N}(z^{2}).\label{CORpoly}%
\end{equation}
Then%
\[
\nu_{t}=\nu_{0}^{\hat{\oplus}\frac{1}{1-t}}.
\]

\end{theorem}

This result is essentially Theorem 4.8 in \cite{COR}. Although the statement
of \cite[Theorem 4.8]{COR} requires that $\nu_{0}$ be the Brown measure of
an $R$-diagonal element, the proof does not use this assumption, provided one
takes \cite[Eq. (4.7)]{COR} as the definition of the fractional free convolution.

Now, if one's goal is to connect the differentiation flow on polynomials to an
operation in free probability, the preceding result certainly achieves the
goal. Nevertheless, the occurrence of $P_{t}^{N}(z^{2})$ rather than
$P_{t}^{N}(z)$ on the right-hand side of (\ref{CORpoly}) is surprising. One
could then ask whether there is a \textit{different} operation on polynomials
that would correspond directly to the fractional free convolution, without the
need for this squaring operation. We will show that this operation is
essentially the $a=-1,$ $b=1$ differential flow.

To motivate this idea, note that%
\[
\frac{1}{2z}\frac{d}{dz}P(z^{2})=P^{\prime}(z^{2}).
\]
Thus, applying ordinary differentiation to $P_{t}^{N}$ is the same as applying
the operator $\frac{1}{2z}\frac{d}{dz}$ to $\hat{P}_{t}^{N}.$ Now, $\hat
{P}_{t}^{N}$ is, by construction, an even polynomial, so applying $\frac{1}%
{z}\frac{d}{dz}$ repeatedly will give another even polynomial; no negative
powers of $z$ will be generated. If we apply $\frac{1}{z}\frac{d}{dz}$ to a
general polynomial, we will get negative powers. If, however, we throw away
those negative powers---as in Definitions \ref{DAB.def} and
\ref{FlowGeneralAB.def}---the evolution of the zeros will be similar to
applying $\frac{1}{z}\frac{d}{dz}$ to an even polynomial.

There is, however, one more point to consider, which is that $\hat{P}_{t}^{N}$
has twice the degree of $P_{t}^{N}.$ Thus, the \textquotedblleft
time\textquotedblright\ in the flow is computed differently by a factor of 2.
That is, if we apply $\frac{1}{z}\frac{d}{dz}$ to a polynomial of degree $N$
rather than $2N,$ we need to rescale $t$ to $t/2.$ The preceding discussion
motivates the following result.

\begin{theorem}[Repeated action of $z^{-1}d/dz$ in terms of $\hat\oplus$]
\label{ConvFlow.thm}Let $P_{0}^{N}$ be a polynomial with independent
coefficients satisfying Assumption \ref{InitialPoly.assumption} and let
$P_{t}^{N}$ be the polynomial obtained by applying the flow in Definition
\ref{FlowGeneralAB.def} with $a=-1$ and $b=1,$ for $t<t_{\max}=1/2.$ That is to say,
$P_t^N$ is obtained by repeatedly applying the operator
\[ \frac{1}{z}\frac{d}{dz}
\]
to $P_0^N$ and then discarding any negative terms that arise. Let
$\mu_{t}$ denote the limiting root distribution of $P_{t/2}^{N}$ (note the
factor of 2). Then%
\[
\mu_{t}=\mu_{0}^{\hat{\oplus}\frac{1}{1-t}}%
\]
for all $t$ with $0\leq t<1.$ In particular, if $P_{0}^{N}$ is a Weyl
polynomial, the limiting root distribution of $P_{t/2}^{N}$ is the uniform
probability measure on a disk of radius $\sqrt{1-t}.$
\end{theorem}

Note that we have (1) changed from the differentiation flow to the
differential flow with $a=-1$ and $b=1$ and (2) rescaled the fractional free
convolution. The advantages of these changes are (1) that we no longer have to
square the variable in the polynomial $P_{t}^{N}$ and (2) we do not have to
rescale the variable of the polynomial by a constant as in Theorem 4.8 of
\cite{COR}. But there still remains a factor of 2 scaling in the time variable
in our Theorem \ref{ConvFlow.thm}. In order to keep the same scaling in the
variable $t$, one could use the differential flow with $a=-1/2, b=1/2$ (Remark
\ref{scaling.remark}). The current approach, however, avoids fractional
derivatives and is more easily motivated.

\begin{proof}
Let $g_0$ be the exponential profile of $P_{0}^{N}$ and let $\mu_{0}$ be the
limiting root distribution. Then the function
\[
r(\alpha):=e^{-g_0^{\prime}(\alpha)},
\]
where $g_0^{\prime}$ is the left-derivative of $g_0,$ is the radial quantile
function of $\mu_{0}.$ That is to say,%
\begin{equation}
r(\alpha)=\inf_{r}\{r\left\vert \mu_{0}(D_{r})\right.  \geq\alpha
\},\label{quantileDef}%
\end{equation}
where $D_{r}$ is the closed disk of radius $r.$ We then apply this result to
the polynomial $P_{t}^{N},$ whose exponential profile is $h_{t},$ in the case
$a=-1$ and $b=1,$ with $t$ replaced by $t/2.$ Using Proposition
\ref{PNprofile.prop} we compute that%
\begin{equation}
e^{-h_{t/2}^{\prime}(\alpha)}=\sqrt{1-t}e^{-g_{0}^{\prime}(\alpha
(1-t)+t)}\sqrt{\frac{\alpha}{\alpha(1-t)+t}}.\label{htOver2}%
\end{equation}

We now compare (\ref{htOver2}) to Eq. (4.7) in \cite{COR}, which computes the
quantile function under the fractional free convolution flow. If $\lambda$ and
$x$ there corresponding to $1-t$ and $\alpha$ here, we see that (\ref{htOver2}%
) agrees with \cite[Eq. (4.7)]{COR}, up to a factor of $1-t,$ after
recognizing that the function $e^{-g^{\prime}(\alpha)}$ in Theorem
\ref{KZ.thm} is the inverse of the radial CDF. The factor of $1-t$ accounts
for the conversion from the original fractional free convolution to its
rescaled version.

The claim about the Weyl case then follows from Example \ref{LOstable.example}, with
$a=-1,$ $b=1,$ and $\beta=1/2,$ after changing $t$ to $t/2.$
\end{proof}

\begin{theorem}
[Push-forward theorem for the fractional free convolution]\label{ConvPush.thm}%
Let $\mu$ be a compactly supported, rotationally invariant probability measure on $\mathbb C$
and define%
\[
\alpha_{0}(r)=\mu(D_{r}),
\]
where $D_{r}$ is the closed disk of radius $r.$ Assume that $\alpha_{0}$ is
continuous and that $\mu$ has no mass at the origin. Let $A_{t}$ be the set of
$w$ with $\alpha_{0}(\left\vert w\right\vert )>t$ and define a map
$T_{t}:A_{t}\rightarrow\mathbb{C}$ by
\[
T_{t}(w)=w\sqrt{\frac{\alpha_{0}(|w|)-t}{\alpha_{0}(|w|)}}.
\]
Then%
\begin{equation}
\mu^{\hat{\oplus}\frac{1}{1-t}}=\frac{1}{1-t}(T_{t})_{\#}(\left.
\mu\right\vert _{A_{t}}),\label{twoMeasures}%
\end{equation}
where $(T_{t})_{\#}$ denotes push-forward by $T_{t}.$
\end{theorem}

The transport map $T_{t}$ comes from the relation (\ref{htOver2}), after
identifying $\alpha_{0}(\left\vert w\right\vert )$ in the definition of
$T_{t}$ with $\alpha(1-t)+t$ in (\ref{htOver2}). We may consider the example
of the uniform probability measure on the unit disk. In that case, $\alpha
_{0}(r)=r^{2}$, so $A_{t}$ is the complement of the disk of radius $\sqrt{t}$
and
\[
T_{t}(re^{i\theta})=e^{i\theta}\sqrt{r^{2}-t}.
\]
The theorem can be verified directly in this case by noting that a uniform
measure on a disk has the property that $r^{2}$ (the square of the magnitude
of a point) is uniformly distributed. The map $T_{t}$ changes the square of
the magnitude by $r^{2}\mapsto r^{2}-t,$ taking a uniform measure on $[0,1]$
to a uniform measure on $[0,1-t].$

As in Remark \ref{pushAssumptions.remark}, Theorem \ref{ConvPush.thm} cannot
hold as stated without the assumption that $\alpha_{0}$ is continuous.

\begin{proof}
If $\mu$ is the limiting root distribution of a sequence of polynomials as in
Theorem \ref{KZ.thm}, then the result follows from Theorem \ref{ConvFlow.thm},
together with the $a=-1,$ $b=1$ case of Corollary \ref{push.cor}, with $t$
replaced by $t/2.$ But we can also check the result directly, as follows.

Let $r_{t}$ be the largest radius such that $\alpha_{0}(r_{t})=t$, so that
$A_{t}$ is the set of $w$ with $\left\vert w\right\vert >r_{t}.$ Then we note
that the function%
\[
\mathcal{T}_{t}(r):=\left\vert T_{t}(re^{i\theta})\right\vert =r\sqrt
{\frac{\alpha_{0}(r)-t}{\alpha_{0}(r)}}%
\]
is continuous and strictly increasing on $(r_{t},\infty)$ and tends to
infinity as $r\rightarrow\infty.$ (The function $r\mapsto r$ is strictly
increasing and tending to infinity and the function $r\mapsto\sqrt{(\alpha
_{0}(r)-t)/\alpha_{0}(r)}$ is positive and nondecreasing.)

We let $\gamma_{t}$ be the pushed-forward measure on the right-hand side of
(\ref{twoMeasures}). Pick $\beta$ with $t<\beta\leq1$ and let $r$ be the
smallest radius for which $\alpha_{0}(r)=\beta.$ Then look at the closed
annulus $E_{r}^{t}$ with inner radius $r_{t}$ and outer radius $r.$ The
measure $\mu_{0}$ assigns this annulus mass $\beta-t.$ Now, the map $T_{t}$
sends $E_{r}^{t}$ injectively to a \textit{disk} $D(\mathcal{T}_{t}(r))$ of
radius $\mathcal{T}_{t}(r)$ and the preimage of this disk is again $E_{r}%
^{t}.$ Thus, using (\ref{twoMeasures}), we see that
\[
\gamma_{t}[(D(\mathcal{T}_{t}(r))]=\alpha:=\frac{\beta-t}{1-t}.
\]
Furthermore, since $r$ is the smallest radius with $\alpha_{0}(r)=\beta,$ we
see that $\mathcal{T}_{t}(r)$ is the smallest radius such that $\gamma
_{t}[D(\mathcal{T}_{t}(r))]$ equals $\alpha.$ That is, the quantile function
(as in (\ref{quantileDef})) of $\gamma_{t}$ at $\alpha$ equals $r.$

Meanwhile, Eq. (4.7) of \cite{COR} tells us the quantile function of
$\mu^{\oplus\frac{1}{1-t}}$ in terms of the quantile function of $\mu.$ We
take $\lambda=1-t$ and $x=\alpha$ in Eq. (4.7) and multiply by a factor of
$\sqrt{1-t}$ to account for our scaling of the fractional free convolution.
After simplifying, we find that the smallest radius for which
$\mu^{\hat{\oplus}\frac{1}{1-t}}$ has measure $\alpha$ is%
\[
r\sqrt{\frac{\alpha_{0}(r)-t}{\alpha_{0}(r)}}=\mathcal{T}%
_{t}(r),
\]
where $r$ is computed in terms of the radial quantile function of $\mu_{0}$ as
the smallest radius for which
\[
\alpha_{0}(r)=\alpha(1-t)+t=\beta,
\]
that is, the same radius $r$ as in the previous paragraph. Thus, the quantile
function of the pushed-forward measure $\gamma_{t}$ agrees with the quantile
function of $\mu^{\oplus\frac{1}{1-t}},$ for every $\alpha=(\beta-t)/(1-t)$
between 0 and 1.
\end{proof}

\section{Free probability interpretation: multiplicative convolution\label{freeProbInterpret.sec}}

In this section, we give a free probability interpretation of the evolution of
zeros of random polynomials under the differential flows in Definition
\ref{FlowGeneralAB.def}, in terms of multiplication of an $R$-diagonal
operator by a freely independent \textquotedblleft transport
operator.\textquotedblright\ In Section \ref{connectTransportCOR.sec}, we show
how this approach relates to the free probability interpretation of Campbell,
O'Rourke, and Renfrew \cite{COR} in the case of repeated differentiation. (See
also Section \ref{fracConvolve.sec}.) We refer again to the monographs of Nica and Speicher \cite{NicaSpeicherLecture} and 
Mingo and Speicher \cite{MingoSpeicher} for basic information about free probability and $R$-diagonal 
operators.

\subsection{The transport operator}

In this subsection, we assume that the limiting root distribution $\mu_{0}$ of
$P_{0}^{N}$ can be expressed as the Brown measure of an $R$-diagonal operator
$A.$ Then, under some assumptions on $a$ and $b,$ we show that the limiting
root distribution $\sigma_{t}$ of $Q_{t}^{N}$ can be expressed as the Brown
measure of $AR_{t}^{a,b},$ for a certain $R$-diagonal \textquotedblleft
transport operator\textquotedblright\ assumed to be freely independent of $A.$
In the next subsection, we will reformulate this result in terms of a type of
free multiplicative convolution for radial measures, at which point we can
drop the restrictions on $a$ and $b$ and the requirement that $\mu_{0}$ be the
Brown measure of an $R$-diagonal operator.

The main idea behind our results is the fact that the exponential profile
$g_{t}$ of $Q_{t}^{N}$ in Theorem \ref{QNprofile.thm} is the sum of the
exponential profile $g_{0}$ of $P_{0}^{N}$ and another, explicit, function.

\begin{theorem}
\label{generalQ.thm}Assume that $b>0$, that $a\geq-b,$ and, if $a<b$, that, $t<t_{\max}=1/(b-a)$. Assume
that the limiting root distribution $\mu_{0}$ of $P_{0}^{N}$ is the Brown
measure of an $R$-diagonal element $A.$ Then we have the following results.

\begin{enumerate}
\item After enlarging the von Neumann algebra as necessary, we can find an
$R$-diagonal \textbf{transport operator} $R_{t}^{a,b}$ that is freely
independent of $A$ such that the limiting root distribution $\sigma_{t}$ of
$Q_{t}^{N}$ is the Brown measure of $AR_{t}^{a,b}$:%
\[
\sigma_{t}=\mathrm{Brown}(AR_{t}^{a,b}).
\]

\item The $R$-diagonal element $R_{t}^{a,b}$ has the property that its Brown
measure is the measure $\sigma_{t}$ in the case that $P_{0}^{N}$ is a Kac
polynomial (with exponential profile $g\equiv 0$). This property determines $R_{t}^{a,b}$ uniquely up to $\ast$-distribution.

\item The element $R_{t}^{a,b}$ may be characterized by the radial quantile function (as in \eqref{quantileDef}) of 
its Brown measure, namely
\begin{equation}\label{rtabQuant}
r_{t}^{a,b}(\alpha)=
\begin{cases}
\left(\frac{\alpha+t(a-b)}{\alpha} \right)^{\frac{b}{b-a}} & a\neq b\\
\exp\left(-\frac{bt}{\alpha} \right) & a=b
\end{cases}.
\end{equation}
Here we set $r_{t}^{a,b}(\alpha)=0$ for $\alpha<t(b-a)$ in the case $a<b.$

\item If $2b/(b-a)$ is a positive integer---in which case $a$ must be less
than $b$---then $R_{t}^{a,b}$ may be computed as%
\begin{equation}
R_{t}^{a,b}=(up)^{\frac{2b}{b-a}},\label{upPower}%
\end{equation}
where $u$ is a Haar unitary, $p$ is a self-adjoint projection freely
independent of $u,$ and
\[
\mathrm{tr}(p)=1-t(b-a),
\]
where $\mathrm{tr}$ is the trace on the relevant von Neumann algebra. In
particular, in the case of repeated differentiation ($a=0$ and $b=1$), we have%
\[
R_{t}^{0,1}=(up)^{2},
\]
where $\mathrm{tr}(p)=1-t.$
\end{enumerate}
\end{theorem}

In the case $a<b$ (with $b>0$), the condition $a\geq-b$ is equivalent to
$2b/(b-a)\geq1.$ In this case, the condition $2b/(b-a)\geq1$ will
guarantee that the convolution power in \eqref{nu.Thetat} below is at least 1. 

\begin{remark}
The quantile function $r_t^{a,b}$ in \eqref{rtabQuant} is closely connected to the transport operator $T_t$ in Definition \ref{Tt.def}. Specifically, comparing \eqref{Tt1} to \eqref{rtabQuant}, we find that 
\[
T_t(w)=w\, r_t^{a,b}(\alpha_0(\left\vert w\right\vert))
\]
\end{remark}

In our next result, we compute the law (or spectral measure) $\mu_{\vert R_{t}^{a,b}\vert
^{2}}$ of $\vert R_{t}^{a,b}\vert^{2}$ as explicitly as possible.

\begin{theorem}
\label{transportOps.thm}We describe $\vert R_{t}^{a,b}\vert ^{2}$
in three cases:

\begin{enumerate}
\item The \textbf{degree-decreasing case} $a<b,$ with $b>0$ and $a\geq-b.$
Consider the family of measures parametrized by $\gamma\in\lbrack0,1]$
\begin{equation}
\nu_{\gamma}=\gamma\delta_{0}+(1-\gamma)\delta_{1}\label{nuGamma}%
\end{equation}
which is the spectral distribution of an orthogonal projection with trace
$1-\gamma$. Then
\begin{equation}
\mu_{|R_{t}^{a,b}|^{2}}=\left(  \nu_{t(b-a)}\right)  ^{\boxtimes\frac{2b}%
{b-a}}.\label{nu.Thetat}%
\end{equation}

\item The \textbf{degree-increasing case} $a>b>0.$ There is a family
$\xi_{\gamma},$ $\gamma>0,$ of $\boxtimes$-infinitely divisible probability
measures on $[0,1]$ such that
\[
\mu_{|R_{t}^{a,b}|^{2}}=\left(  \xi_{t(a-b)}\right)  ^{\boxtimes\frac{2b}%
{a-b}}.
\]
The measures $\xi_{\gamma}$ are described explicitly in (\ref{dXi}) in Section
\ref{transportProofs.sec}. 

\item The \textbf{degree-preserving case} $a=b>0.$ Then there is one-parameter
$\boxtimes$-semigroup $\eta_{\gamma}$ of $\boxtimes$-infinitely divisible
measures such that
\[
\mu_{|R_{t}^{b,b}|^{2}}=\eta_{tb}.
\]

\end{enumerate}
\end{theorem}

\begin{remark}
In the case $a=b,$ with $b>0,$ the $R$-diagonal operator $R_{t}^{b,b}$ can be
computed as the limit as $a$ approaches $b$ from below of the operator in
(\ref{upPower}), with $a$ chosen so that $2b/(b-a)$ is an integer. In this
limit, the trace of $p$ approaches $1$ while the power $2b/(b-a)$ approaches
$+\infty.$ It is therefore natural to think of $R_{t}^{b,b}$ as an
$R$-diagonal form of a free multiplicative Poisson process. Indeed, the law of
$\vert R_{t}^{b,b}\vert ^{2}$ can be computed as the limit of the
measures in (\ref{nu.Thetat}) as $a\rightarrow b^{-},$ so that $\vert
R_{t}^{b,b}\vert ^{2}$ can be thought of as a free multiplicative
Poisson process consisting of non-negative self-adjoint operators. 
\end{remark}

\begin{remark}
The counterparts of Theorem \ref{transportOps.thm} and Theorem \ref{generalQ.thm} for polynomials with non-negative real roots have been established in \cite[\S 4]{JKM1}, when $a$ and $b$ are non-negative integers. More precisely, take a sequence of polynomials with non-negative real roots with limiting root distribution $\nu$ on $[0,\infty)$. Then, applying $z^{-\lceil N(t(a-b))\rceil} \big(z^a(\frac{d}{dz})^b\big)^{tN}$ to that sequence (analogous to our $Q_t^N$) yields a limiting root distribution $\nu\boxtimes \nu_{a,b;t}$, computed using the ordinary free multiplicative convolution $\boxtimes$ of measures on $[0,\infty)$. (The limiting measure is real counterpart of the measure $\sigma_t$ from Theorem \ref{generalQ.thm}.) The measures $\nu_{a,b;t}$ are directly related to the transport operators in our Theorem \ref{transportOps.thm} by the relation $$\nu_{a,b;t}^{\boxtimes 2} =\mu_{|R_t^{a,b}|^2}.$$.
\end{remark}

The proofs of Theorems \ref{generalQ.thm} and \ref{transportOps.thm} will be
provided in Section \ref{transportProofs.sec}.

\subsection{Free multiplicative convolution for radial measures}

Suppose $\mu$ is a compactly supported radial probability measure on
$\mathbb{C}.$ Define the radial cumulative distribution function $\alpha_{\mu}$ of
$\mu$ as
\[
\alpha_{\mu}(r)=\mu(\{\left\vert z\right\vert \leq r\}), \quad r\ge 0.
\]
Then define the radial quantile function $r_{\mu}$ of $\mu$ as%
\[
r_{\mu}(\alpha)=\inf\{r\ge 0|\alpha_{\mu}(r)\geq\alpha\},\quad\alpha\in\lbrack0,1].
\]
If $\alpha_{\mu}$ is continuous and strictly increasing on some interval
\thinspace$\lbrack r_{\textrm{in}},r_{\textrm{out}}],$ with $\alpha_\mu(r_{\textrm{in}})=0$ and $\alpha_{\mu}(r_{\textrm{out}})=1,$ then $r_{\mu}:[0,1]\rightarrow[r_{\mathrm{in}},r_{\mathrm{out}}]$ is the inverse
function to $\alpha_{\mu}.$

In general, $r_{\mu}$ is nondecreasing and left-continuous on $[0,1],$ with
$r_{\mu}(0)=0$. Any function $r(\cdot)$ with these properties is the quantile
function of a unique compactly supported radial probability measure $\mu
$---namely the measure whose radial part is the push-forward of the uniform
measure on $[0,1]$ under $r(\cdot)$. Thus, if $r_{\mu_{1}}$ and $r_{\mu_{2}}$ are the
quantile functions of compactly supported radial probability measures $\mu
_{1}$ and $\mu_{2},$ the product function $r_{\mu_{1}}r_{\mu_{2}}$ is the
quantile function of a unique compactly supported radial probability measure.

Suppose $P_{0}^{N}$ is a family of random polynomials with independent
coefficients with exponential profile $g,$ satisfying Assumptions
\ref{gProperties.assumption} and \ref{KZcoeff.assumption}. Let $\mu$ be the
limiting root distribution of $P_{0}^{N}.$ Then the quantile function $r_{\mu
}$ of $\mu$ is computed as%
\[
r_{\mu}(\alpha)=e^{-g^{\prime}(\alpha)},
\]
where $g^{\prime}(\alpha)$ is taken to be $+\infty$ for $\alpha<\alpha_{\min
}.$ (See Theorem \ref{KZ.thm}.) In this setting, multiplying the quantile
functions is equivalent to adding the exponential profiles.

Our next result says that multiplying freely independent $R$-diagonal operators corresponds to 
multiplying the radial quantile functions of their Brown measures. 

\begin{proposition}
Suppose $A_{1}$ and $A_{2}$ are freely independent $R$-diagonal operators and
let $r_{1}$ and $r_{2}$ be the radial quantile functions of the Brown measures
of $A_{1}$ and $A_{2},$ respectively. Then $A_{1}A_{2}$ is $R$-diagonal and
the quantile function of its Brown measure is $r_{1}r_{2}.$
\end{proposition}

\begin{proof}
Consider first a single $R$-diagonal element $A$ with Brown measure $\mu_{A}.$
We write $A=uh$ with $u$ a Haar unitary, $h\geq0,$ and $u$ and $h$ freely
independent. We let $\delta$ be the mass of the law of $h$ at 0. Then by
\cite[Theorem 4.4(iii)]{HL} the Brown measure of $A$ also has mass $\delta$ at
the origin, so that $r$ is zero on $[0,\delta]$ and strictly positive on
$(\delta,1].$ Then by the same theorem, the quantile function of $\mu_{A}$ is
computed as
\begin{equation}
r_{\mu_{A}}(\alpha)=\frac{1}{\sqrt{S_{A^{\ast}A}(\alpha-1)}},\quad
\alpha>\delta. \label{HLStransform}%
\end{equation}
where $S_{A^{\ast}A}$ is the $S$-transform of $A^{\ast}A=h^{2}.$ (See Section \ref{transportProofs.sec} for the definition of $S$.)

Now consider $A_{1}$ and $A_{2},$ written as $u_{1}h_{1}$ and $u_{2}h_{2},$
with $\mu_{h_{1}}$ and $\mu_{h_{2}}$ having mass $\delta_{1}$ and $\delta_{2}$
at the origin. According to \cite[Proposition 3.6(ii)]{HL}, the law of
$\left\vert A_{1}A_{2}\right\vert ^{2}$ is the free multiplicative convolution
of the law of $\left\vert A_{1}\right\vert ^{2}$ with the law of $\left\vert
A_{2}\right\vert ^{2}.$ Then by \cite[Lemma 6.9]{BVIU}, the law of $\left\vert
A_{1}A_{2}\right\vert ^{2}$---and therefore also the Brown measure of
$A_{1}A_{2}$---has mass exactly $\max(\delta_{1},\delta_{2})$ at the origin.
Thus, $r_{\mu_{A_{1}A_{2}}}=r_{\mu_{A_{1}}}r_{\mu_{A_{2}}}=0$ on
$[0,\max(\delta_{1},\delta_{2})].$ Once this is established, we use
(\ref{HLStransform}) and the multiplicativity of the $S$-transforms
(\cite[Theorem 2.6]{VoiculescuMult} or \cite[Corollary 6.6]%
{BVIU}) to conclude that $r_{\mu_{A_{1}A_{2}}}=r_{\mu_{A_{1}}%
}r_{\mu_{A_{2}}}$ on $(\max(\delta_{1},\delta_{2}),1].$
\end{proof}

In light of the proposition, it is natural to define a free multiplicative
convolution $\otimes$ (the circled symbol indicating isotropy, not to be confused with the product measure) on compactly supported
radial probability measures as follows.

\begin{definition}
\label{radialConv.def}If $\mu_{1}$ and $\mu_{2}$ are compactly supported
radial probability measures with radial quantile functions $r_{1}$ and
$r_{2},$ we define
\[
\mu_{1}\otimes\mu_{2}%
\]
as the compactly supported radial probability measure with quantile function
$r_{1}r_{2}.$
\end{definition}

\begin{example}
Fix a positive integer $N$. Suppose $\mu_{1}$ assigns mass $1/N$ to circles of
radii $s_{1},\ldots,s_{N}$ with $s_{1}<\cdots<s_{N}$ and $\mu_{2}$ assigns
mass $1/N$ to circles of radii $t_{1},\ldots,t_{N}$ with $t_{1}<\cdots<t_{N}.$
Then $\mu_{1}\otimes\mu_{2}$ assigns mass $1/N$ to each circle of radius
$s_{i}t_{i}$ with $1\leq i\leq N.$
\end{example}

\begin{proof}
In this case, $r_{1}$ takes the value $s_{i}$ on the interval $((i-1)/N,i/N]$
and $r_{2}$ takes the value $t_{i}$ on $((i-1)/N,i/N].$ Thus, $r_{1}r_{2}$
takes the value $s_{i}t_{i}$ on $((i-1)/N,i/N],$ which corresponds to the
claimed value of $\mu_{1}\otimes\mu_{2}.$
\end{proof}

We now reformulate Theorem \ref{generalQ.thm} in the language of the radial
free multiplicative convolution.

\begin{theorem}
\label{ConvolutionTransport.thm}Assume $b>0$ and $t<t_{\max}$ and let $\mu
_{0}$ be the limiting root distribution of $P_{0}^{N}.$ Let $\rho_{t}^{a,b}$
be the radial measure whose quantile function $r_{t}^{a,b}$ is given by \eqref{rtabQuant}. 
Then the limiting root distribution $\sigma_{t}$ of $Q_{t}^{N}$ is given by
\[
\sigma_{t}=\mu_{0}\otimes\rho_{t}^{a,b}
\]
where $\otimes$ is as in Definition \ref{radialConv.def}.

\end{theorem}

In the cases where Theorem \ref{generalQ.thm} is applicable, the measure
$\rho_{t}^{a,b}$ is simply the Brown measure of the transport operator
$R_{t}^{a,b}.$ But Theorem \ref{ConvolutionTransport.thm} eliminates the
assumptions that $a\geq-b$ and that $\mu_{0}$ is the Brown measure of an
$R$-diagonal element in Theorem \ref{generalQ.thm}.

The above findings also give rise to a relation between the free radial additive convolution $\oplus$ and its multiplicative counterpart $\otimes$ in Definition \ref{radialConv.def}.
\begin{corollary}
Let $\mu_0$ be a radial distribution obtained from an exponential profile satisfying Assumption \ref{gProperties.assumption}. Let $\hat\rho_t$ be the Brown measure of $\frac{1}{1-t}up$, where as above $u$ is Haar unitary and $p$ is a free projection of trace $1-t$. For any $t\in [0,1)$ it holds
\begin{align}\label{eq:relation}
\mu_0\otimes\hat \rho_t=t\delta_0+(1-t)\mu^{\oplus \frac{1}{1-t}}.
\end{align}
\end{corollary}

The measure $\hat\rho_t$ is the radial, or $R$-diagonalized, version of the distribution $t\delta_0+(1-t)\delta_{\frac{1}{1-t}}$ of the positive element $\tfrac1 {1-t} p$, or $\mathcal H(t\delta_0+(1-t)\delta_{\frac{1}{1-t}})=\hat \rho_t$ in notation of K\"{o}sters and Tikhomirov \cite{KostersTikhomirov}. Having this in mind, Equation \eqref{eq:relation} is the isotropic analogue of the Equation (14.13) in \cite{NicaSpeicherLecture} stating the following:
For any (compactly supported) distribution $\mu$ on $\mathbb R$ and any $t\in[0,1)$ it holds
$$\mu\boxtimes\big(t\delta_0+(1-t)\delta_{\frac{1}{1-t}}\big)=t\delta_0+(1-t)\mu^{\boxplus \frac{1}{1-t}}.$$

\begin{proof}
We consider the case $a=-1, b=1$ of Theorem \ref{ConvFlow.thm}, which says that 
$$\mu_t=\mu_0^{\hat\oplus \frac{1}{1-t}}=\mu_0^{\oplus \frac{1}{1-t}}\big(\tfrac{1}{1-t}\cdot \big).$$ Thus, the associated measure $\sigma_t=(1-t)\mu_t+t\delta_0$ is computed as $$\sigma_t=t\delta_0+(1-t)\mu_0^{\oplus\frac{1}{1-t}}\big( \tfrac{1}{1-t}\cdot\big) .$$
On the other hand, Theorem \ref{ConvolutionTransport.thm} implies $\sigma_t=\mu_0\otimes\rho_t^{-1,1}$, where $\rho_t$ is the Brown measure of $R_t^{-1,1}=up$ for $u$ Haar unitary and a free projection $p$ of trace $1-t$. The claim follows from the observation that the push-forward of a multiplication by $1-t$ may only act on one factor in $\otimes$, here $\rho_t^{-1,1}$.
\end{proof}

\subsection{Proofs of the main results\label{transportProofs.sec}}

If $\mu$ is a compactly supported probability measure on $[0,\infty),$ the
$\psi$-transform of $\mu$ is the function given by
\[
\psi(z)=\int_{0}^{\infty}\frac{zt}{1-zt}~d\mu(t)
\]
for $z$ outside the support of $\mu.$ The $\psi$-transform of $\mu$ is related
to its Cauchy transform $m$ as $\psi(z)=m(1/z)/z-1.$ Then the $S$-transform of
$\mu$ is the function satisfying%
\begin{equation}
\psi\left(  \frac{z}{z+1}S(z)\right)  =z,\label{PsiAndS}%
\end{equation}
for $z$ in a domain in $\mathbb{C}$ that contains an interval of the form
$(-\varepsilon,0),$ $\varepsilon>0.$ If $x$ is a non-negative self-adjoint
element, we let $S_{x}$ denote the $S$-transform of the law of $x.$

We first prove Theorem \ref{ConvolutionTransport.thm}.

\begin{proof}
[Proof of Theorem \ref{ConvolutionTransport.thm}]We use that the quantile function of the limiting
root distribution of a random polynomial with exponential profile $g$
(satisfying Assumptions \ref{gProperties.assumption} and
\ref{KZcoeff.assumption}) is given as
\[
r(\alpha)=e^{-g^{\prime}(\alpha)},
\]
where we interpret $g^{\prime}(\alpha)$ as equaling $+\infty$ when $\alpha$ is
less than the constant $\alpha_{\min}$ in Assumption
\ref{gProperties.assumption}. (Recall Theorem \ref{KZ.thm}.)

We then note that, by Theorem \ref{QNprofile.thm}, the exponential profile
$g_{t}$ of $Q_{t}^{N}$ is the sum of the exponential profile $g_{0}$ of
$P_{0}^{N},$ and another, explicit term $G_{t}^{a,b},$ computed as%
\[
G_{t}^{a,b}(\alpha)=\left\{
\begin{array}
[c]{ccc}%
\frac{b}{a-b}\{[\alpha+t(a-b)]\log[\alpha+t(a-b)]-\alpha\log\alpha\}-bt\  &  &
a\neq b\\
bt\log\alpha &  & a=b
\end{array}
\right.  .
\]
Thus, the quantile function of $\sigma_{t}$ will be the product of the
quantile function $r_{0}$ of $\mu_{0}$ and the function $r_{t}^{a,b}%
(\alpha):=e^{-(G_{t}^{a,b})^{\prime}(\alpha)},$ where a computation shows that
$e^{-(G_{t}^{a,b})^{\prime}(\alpha)}$ coincides with the function $r_t^{a,b}$ in \eqref{rtabQuant}.
Thus, by Definition \ref{radialConv.def}, $\sigma_{t}$ is the radial convolution
of $\mu_{0}$ with the measure $\rho_{t}^{a,b}$ having quantile function
$r_{t}^{a,b},$ as claimed.
\end{proof}

We now turn to the proofs of Theorems \ref{generalQ.thm} and
\ref{transportOps.thm}.

\begin{proof}
[Proofs of Theorems \ref{generalQ.thm} and \ref{transportOps.thm}]If the
measure $\rho_{t}^{a,b}$ with quantile function $r_{t}^{a,b}$ is the Brown
measure of an $R$-diagonal operator $R_{t}^{a,b},$ the $S$-transform of
$\vert R_{t}^{a,b}\vert ^{2}$ can be computed using
(\ref{HLStransform}), giving the formula 
\begin{equation}
S_{\vert R_{t}^{a,b}\vert ^{2}}(z)=%
\begin{cases}
\left(  \frac{z+1}{z+1+t(a-b)}\right)  ^{2b/(b-a)}\quad & a\neq b\\[8pt]%
\exp\left(  \frac{2bt}{z+1}\right)  & a=b.
\end{cases}
 \label{Sformula}%
\end{equation}
The $S$-transform in the $a=b$ case of (\ref{Sformula}) is equivalent to the
$\Sigma$-transform in \cite[Lemma 6.12(ii)]{BVIU}, which corresponds to the
mass at $+\infty$ in the L\'{e}vy--Hin\v{c}in decomposition in \cite[Theorem
6.13(iii)]{BVIU}. 

Our next step is to determine when the functions in
(\ref{Sformula}) are actually $S$-transforms of probability measures on
$[0,\infty).$ We divide our analysis into the three cases in Theorem
\ref{transportOps.thm}.

We start with the \textbf{degree-decreasing case} $a<b,$ with $b>0$ and
$a\geq-b.$ It is easy to verify that for every $\gamma\in\lbrack0,1)$, the
function
\[
s(z)=\frac{z+1}{z+1-\gamma},\quad z\in\mathbb{C}\setminus(-\infty,\gamma-1]
\]
is the $S$-transform of the measure $\nu_{\gamma}$, which is the law of a
projection with trace $1-\gamma$. Then by \cite[Theorem 2.6]{BB}, the function%
\[
S(z)=\left(  \frac{z+1}{z+1-\gamma}\right)  ^{\delta}%
\]
is the $S$-transform of $(\nu_{\gamma})^{\boxtimes\delta}$ for all $\delta
\geq1.$ We apply this result with $\gamma=t(b-a)$ and $\delta=2b/(b-a),$ where
under the stated assumptions on $a$ and $b$---and assuming $t<t_{\max
}=1/(b-a)$---we have $\gamma\in(0,1)$ and $\delta\geq1.$ 

We now turn to the \textbf{degree-increasing case} $a>b,$ with $b>0.$ We then
claim that for all $\gamma>0,$ the function
\begin{equation}
s(z)=\frac{z+1+\gamma}{z+1}=1+\frac{\gamma}{z+1}\label{Sgamma}%
\end{equation}
is the $S$-transform of a $\boxtimes$-infinitely divisible probability measure
$\xi_{\gamma}$ on $[0,1].$ This claim follows from \cite[Theorem 6.13]{BVIU},
after correcting a minor typographical error there. (The domain $\mathbb{C}%
\setminus(0,1)$ should be $(\mathbb{C}\setminus\mathbb{R})\cup(-1,0),$ which
is a weakening of the corresponding condition in \cite[Theorem 7.5(ii)]%
{BVPJM}.) 

We then claim that the $S$-transform in (\ref{Sgamma}) has the
following L\'{e}vy--Hin\v{c}in representation (as in \cite[Theorem
6.13(iii)]{BVIU}):
\[
\log\left(  s\left(  \frac{z}{1-z}\right)  \right)  =\int_{\frac{1+\gamma
}{\gamma}}^{\infty}\frac{1+\lambda z}{z-\lambda}\frac{c_{\gamma}}%
{1+\lambda^{2}}~d\lambda,
\]
where%
\[
c_{\gamma}=\frac{1}{2}\log(1+2\gamma+2\gamma^{2}),
\]
as may be verified by direct calculation. Then by \cite[Theorem 2.6]{BB}, the
$S$-transform in the first line of (\ref{Sformula}), in the case $a>b,$ is the
$S$-transform of $\left(  \xi_{t(a-b)}\right)  ^{\boxtimes\frac{2b}{a-b}},$ as
claimed. Note that since $\xi_{t(a-b)}$ is infinitely divisible, the exponent
$2b/(a-b)$ is allowed to be less than 1, provided it is positive. 

We now compute the measure $\xi_{\gamma}$ explicitly. Using the formula for
(\ref{Sgamma}) for $s(z),$ we can solve for $\psi$ in (\ref{PsiAndS}) and then
compute $m$ as
\[
m(z)=\frac{1-\gamma+\sqrt{(1+\gamma)^{2}-\frac{4\gamma}{z}}}{2(z-1)},
\]
where we use the principal branch of the square root for $z\in\mathbb{C}$
outside the interval $[0,4\gamma/(1+\gamma)^{2}]$ in $\mathbb{R}.$ It is
straightforward to check that $m$ is the Cauchy transform of a probability
measure $\xi_{\gamma}$ on $\mathbb{R},$ computed explicitly by the Stieltjes
inversion formula as%
\begin{equation}
d\xi_{\gamma}(x)=\max(0,1-\gamma)\delta_{1}+\mathbf{1}_{(0,4\gamma
/(1+\gamma)^{2})}\frac{1}{2\pi}\frac{\sqrt{4\gamma-(1+\gamma)^{2}x}%
}{(1-x)\sqrt{x}}~dx.\label{dXi}%
\end{equation}

Finally, we consider the \textbf{degree-preserving case} $a=b,$ with $b>0.$ In
that case, we may again apply \cite[Theorem 6.13]{BVIU} to show that%
\[
S(z)=\exp\left(  \frac{\gamma}{z+1}\right)
\]
is the $S$-transform of a $\boxtimes$-infinitely divisible measure
$\eta_{\gamma}.$ This $S$-transform is the $t\rightarrow1$ limit of the
$S$-transform in \cite[Lemma 7.2]{BVPJM}, which is the $S$-transform of a
\textquotedblleft free multiplicative Poisson\textquotedblright\ random
variable. The $t\rightarrow1$ limit causes the support of the measure
$\eta_{\gamma}$ to extend all the way to 0. 
\end{proof}

\subsection{Connection to the results of Campbell, O'Rourke, and
Renfrew\label{connectTransportCOR.sec}}

Our results in this section are not directly comparable to those in the paper
\cite{COR} of Campbell, O'Rourke, and Renfrew, because (as discussed in
Section \ref{fracConvolve.sec}), \cite{COR} assumes that the limiting root
distribution of $P_{0}^{N}(z^{2})$ is the Brown measure of an $R$-diagonal
operator, while we assume that the limiting root distribution of $P_{0}%
^{N}(z)$ is the Brown measure of an $R$-diagonal operator. If we adjust for
this difference of convention and focus on the case of repeated
differentiation ($a=0,$ $b=1$), Theorem \ref{generalQ.thm} may be restated as
follows. 

\begin{theorem}
[Equivalent form of Theorem \ref{generalQ.thm} in the case $a=0,$ $b=1$]Take
$a=0$ and $b=1,$ and assume that the limiting root distribution of $P_{0}%
^{N}(z^{2})$ is the Brown measure of an $R$-diagonal operator $A.$ Then the
limiting root distribution of $Q_{t}^{N}(z^{2})$ is the Brown measure of the
operator
\[
Aup,
\]
where $u$ is a Haar unitary, $p$ is a projection with trace $1-t,$ and where
$A,$ $u,$ and $p$ are freely independent. 
\end{theorem}

More generally, if we repeat the proof of Theorem \ref{generalQ.thm} using the
polynomials $P_{0}^{N}(z^{2})$ and $Q_{t}^{N}(z^{2})$ in place of $P_{0}%
^{N}(z)$ and $Q_{0}^{N}(z),$ we find in the degree-decreasing case the same
result except that the powers on the right-hand side of (\ref{Sformula}) and
(\ref{upPower}) are $b/(b-a)$ instead of $2b/(b-a).$ We then note that since
$A$ is $R$-diagonal and $p$ is a projection, we have
\[
\mathrm{Brown}(Aup)=\mathrm{Brown}(Ap)=\mathrm{Brown}(pAp).
\]
(The second equality can be established using the argument on p. 350 of
\cite{HL}.) The limiting root distribution of $P_{t}^{N}(z^{2})$ is then
easily seen to be the same as the Brown measure of $pAp,$ viewed as an element
of the compressed von Neumann algebra $p\mathcal{A}p,$ which agrees with the
result of \cite[Theorem 4.8]{COR}, as restated in \cite[Figure 2]{COR}.

The conclusion of the preceding discussion is this: In the repeated
differentiation case and after adjusting for differences of convention, the
\textquotedblleft transport operator\textquotedblright\ approach in our
Theorem \ref{generalQ.thm} is easily seen to be equivalent to the approach in
\cite{COR} using compressions of $R$-diagonal operators (or, equivalently,
fractional free convolution). The transport operator approach lends itself to
the study of general differential flows because the transport operator
$R_{t}^{a,b}$ can be made to depend on parameters $a$ and $b$ in the flow,
whereas in the compression approach, the only available parameter is the trace
of $p.$ 

\section{The PDE analysis\label{PDE.sec}}

In this section, we obtain a PDE\ for (a rescaled version of) the log
potential of the limiting root distribution of polynomial $P_{t}^{N}.$ This
PDE clarifies the push-forward results of Section \ref{PushForward.sec}.
Specifically, we will see that Theorem \ref{push.thm} can be interpreted as a
\textquotedblleft bulk\textquotedblright\ version of the statement that the
\textit{zeros of }$P_{t}^{N}$\textit{ evolve approximately along the
characteristic curves} of the relevant PDE. (Compare the heuristic derivation
of Idea \ref{radialMotion.idea} in Section \ref{newResults.sec}.) We
emphasize, however, that the actual proof of Theorem \ref{push.thm} in Section
\ref{PushForward.sec} is independent of any PDE\ results. The results obtained
in this section are parallel to the results of \cite{HHJK1} for polynomials
evolving according to the heat flow.

If $\mu$ is a compactly supported probability measure on $\mathbb{C},$ we
normalize the log potential $V$ of $\mu$ as%
\begin{equation}
V(z)=\int_{\mathbb{C}}\log(\left\vert z-w\right\vert ^{2})~d\mu(w),\quad
z\in\mathbb{C}.\label{PotentialDef}%
\end{equation}
This definition differs by a factor of 2 from the one used in \cite[Theorem
5.2]{HHJK2}, where $\log(\left\vert z-w)\right\vert $ is used in place of
$\log(\left\vert z-w\right\vert ^{2})$ in (\ref{PotentialDef}). The measure
$\mu$ is then recovered from $V$ as%
\[
\mu=\frac{1}{4\pi}\Delta V,
\]
where $\Delta$ is the distributional Laplacian. If $P$ is a polynomial of
degree $N$ with leading coefficient $a_{N},$ the log potential \thinspace$V$
of the empirical root measure of $P_{0}^{N}$ is easily seen to be%
\begin{equation}
V(z)=\frac{1}{N}\log\left\vert P(z)\right\vert ^{2}-\frac{1}{N}\log\left\vert
a_{N}\right\vert ^{2}.\label{rootMeasure}%
\end{equation}

Now consider the polynomial $P_{t}^{N}$ in Definition \ref{FlowGeneralAB.def},
where the initial polynomial $P_{0}^{N}$ satisfies Assumption
\ref{InitialPoly.assumption}. Let $V_{t}$ denote the log potential of the
limiting root distribution $\mu_{t}$ of $P_{t}^{N}.$ Then define a rescaled
log potential $S_{t}$ of $P_{t}^{N}$ by
\begin{equation}
S_{t}(z)=(1+t(a-b))V_{t}(z)+2g_{t}(1).\label{Snormalization}%
\end{equation}
At least heuristically, $S_{t}$ should be computable as%
\begin{equation}
S_{t}(z)=\lim_{N\rightarrow\infty}\frac{1}{N}\log\left(\left\vert P_{t}%
^{N}(z)\right\vert^2\right) .\label{Sheuristic}%
\end{equation}
The expression (\ref{Sheuristic}) accounts for differences between $S_{t}$ and
$V_{t}$. First, the right-hand side of (\ref{Sheuristic}) is normalized by the
degree $N$ of the original polynomial, rather than by the degree $(1+t(a-b))N$
of the polynomial $P_{t}^{N}.$ Second, the right-hand side of
(\ref{Sheuristic}) does not subtract off a term coming from the leading
coefficient of $P_{t}^{N},$ as in (\ref{rootMeasure}). The normalization in
(\ref{Snormalization}), motivated by (\ref{Sheuristic}), is convenient because
it leads to a nice PDE\ for $S_{t}.$ (Compare the heuristic derivation in
Appendix \ref{PDEderiv.appendix} for the case of repeated differentiation,
where (\ref{Sheuristic}) is used.)

Throughout this section we make the following assumption. Recall from
\eqref{InnerOuter}, \eqref{alpha0} and \eqref{tmax} that we associate to the
exponential profile $g_{t}$ the radii of the support $r_{\mathrm{in}}(t)=\lim
_{\alpha\searrow\alpha^{t}_{\min}}e^{-g_{t}^{\prime}(\alpha)}$,
$r_{\mathrm{out}}(t)=e^{-g_{t}^{\prime}(1)}$, as well as the thresholds
$\alpha_{\min}^{t}=\max(t(b-a),0)$ and $t_{\max}=\frac{1}{b-a}$ if $a<b$,
$t_{\max}=\infty$ if $a\geq b$.

\begin{assumption}
\label{pde.assumption} We assume that $r_{\mathrm{in}}(t)=0$ for all $0\leq
t<t_{\max}$ and that $g_{t}$ is twice continuously differentiable with
$g_{t}^{\prime\prime}<0$ on the interval $(\alpha_{\min}^{t},1)$.
\end{assumption}

If $b>0,$ the assumptions will hold as long as $r_{\mathrm{in}}(0)=0$ and
$g_{0}$ is twice continuously differentiable with $g_{0}^{\prime\prime}<0$ on
$(0,1).$ (By the formulas for $g_{t}$ in Theorem \ref{QNprofile.thm}.)

\begin{theorem}
\label{thePDEforS.thm} Assume Assumptions \ref{KZcoeff.assumption},
\ref{InitialPoly.assumption}, and \ref{pde.assumption}, let $\mu_{t}$ be the
limiting root distribution of $P_{t}^{N},$ let $V_{t}$ be the log potential of
$\mu_{t},$ and then define $S_{t}$ by (\ref{Snormalization}). Then, $S_{t}$
satisfies the PDE%
\begin{equation}
\frac{\partial S_{t}}{\partial t}=\log\left(  \left\vert z^{a}\left(
\frac{\partial S_{t}}{\partial z}\right)  ^{b}\right\vert ^{2}\right)  ,\qquad
z\in\mathbb{C}\backslash\{0\}, \ 0\leq t<t_{\max} . \label{UtPDE}%
\end{equation}

\end{theorem}

We emphasize for the complex variable $z$ as usual,%
\[
\frac{\partial}{\partial z}=\frac{1}{2}\left(  \frac{\partial}{\partial
x}-i\frac{\partial}{\partial y}\right)  ;\quad\frac{\partial}{\partial\bar{z}%
}=\frac{1}{2}\left(  \frac{\partial}{\partial x}+i\frac{\partial}{\partial
y}\right)  ,
\]
whereas $t$ is a \textit{real} variable and so $\partial/\partial t$ is the
ordinary real partial derivative. This situation should be contrasted with the
one in \cite{HHJK2}, where $t$ is a complex variable and the derivatives with
respect to $t$ are Wirtinger derivatives (i.e., Cauchy--Riemann operators).
Since $S_{t}$ is real valued, we may rephrase the right hand side of the PDE
by $\log\left(  z^{a}\left(  \frac{\partial S_{t}}{\partial z}\right)
^{b}\right)  +\log\left(  \bar{z}^{a}\left(  \frac{\partial S_{t}}%
{\partial\bar{z}}\right)  ^{b}\right)  .$

The proof will be given in Section \ref{LogPotPDE.sec}.

\subsection{The log potentials of $Q_{t}^{N}$ and $P_{t}^{N}$}

We first record formulas for the log potentials of the limiting root
distributions of $Q_{t}^{N}$ and $P_{t}^{N}.$

\begin{proposition}
\label{QNtLog.prop}The log potential $W_{t}$ of the limiting root distribution
of $Q_{t}^{N}$ can be computed as%
\begin{equation}
W_{t}(z)=\sup_{\alpha}[2g_{t}(\alpha)-2g_{t}(1)+\alpha\log\left\vert
z\right\vert ^{2}], \label{Wtformula}%
\end{equation}
and the log potential $V_{t}$ of the limiting root distribution of $P_{t}^{N}$
can be computed as%
\begin{equation}
V_{t}(z)=\frac{1}{1+t(a-b)}\left[  \sup_{\alpha}(2g_{t}(\alpha)-2g_{t}%
(1)+\alpha\log\left\vert z\right\vert ^{2})+t(a-b)\log\left\vert z\right\vert
^{2}\right]  , \label{Vtformula}%
\end{equation}
where in both cases, the supremum is taken over $1\ge\alpha\geq\alpha_{\min
}^{t},$ where, as in (\ref{alpha0}), $\alpha_{\min}^{t}$ is $0$ for $a\geq b$
and $\alpha_{\min}^{t}$ is $t(b-a)$ for $a<b.$
\end{proposition}

\begin{proof}
The formula (\ref{Wtformula}) for $W_{t}$ is a direct consequence of
\cite[Theorem 2.8]{KZ}; compare \cite[Theorem 2.2]{HHJK2}, but noting that we
use a different normalization of the log potential here. We then reduce the
formula (\ref{Vtformula}) for $V_{t}$ to the formula (\ref{Wtformula}) for
$W_{t}$ in two cases, starting with the case $a\geq b.$ Recall that $\mu_{t}$
and $\sigma_{t}$ are the limiting root distributions of $P_{t}^{N}$ and
$Q_{t}^{N},$ respectively. In the case $a\geq b,$ the quantity $\alpha_{\min
}^{t}$ equals 0 and the measure $\sigma_{t}$ will have no mass at the origin.
Then the roots of the polynomial $P_{t}^{N}(z)=z^{Nt(a-b)}Q_{t}^{N}(z)$ will
consist of the roots of $Q_{t}^{N}$ together with $Nt(a-b)$ roots at the
origin. From this observation, and keeping in mind that the degree of
$P_{t}^{N}$ is $N(1+t(a-b)),$ we easily see that
\begin{equation}
\mu_{t}=\frac{1}{1+t(a-b)}(\sigma_{t}+t(a-b)\delta_{0}). \label{muAndSigma}%
\end{equation}
Thus, the log potential of $V_{t}$ of $\mu_{t}$ is related to the log
potential $W_{t}$ of $\sigma_{t}$ by%
\begin{equation}
V_{t}(z)=\frac{1}{1+t(a-b)}(W_{t}(z)+t(a-b)\log\left\vert z\right\vert ^{2}).
\label{muSigmaLog}%
\end{equation}
The claimed result then follows by applying (\ref{Wtformula}).

In the case $a<b,$ we claim that (\ref{muAndSigma}) still holds, but with
$a-b$ now being negative. In this case, the measure $\sigma_{t}$ has mass
$t(b-a)$ at the origin. The formula (\ref{muAndSigma}) for $\mu_{t}$ is
obtained by removing this mass and then rescaling the result to be a
probability measure. Then (\ref{muSigmaLog}) follows as in the case $a\geq b.$
\end{proof}

\subsection{The PDE for the normalized log potential of $P_{t}^{N}%
$\label{LogPotPDE.sec}}

Recall that $V_{t}$ is the log potential of the limiting root measure of
$P_{t}^{N}$ and that $S_{t}$ is defined by%
\begin{equation}
S_{t}(z)=(1+t(a-b))V_{t}(z)+2g_{t}(1). \label{StDef}%
\end{equation}
Then by Proposition \ref{QNtLog.prop}, we have%
\begin{equation}
S_{t}(z)=\sup_{\alpha_{\min}^{t}\le\alpha\le1}\big(2g_{t}(\alpha)+\alpha
\log\left\vert z\right\vert ^{2}\big)+t(a-b)\log\left\vert z\right\vert ^{2} .
\label{uth}%
\end{equation}
The proof of Theorem \ref{thePDEforS.thm} will follow the argument in the
proof of Theorem 5.2 in \cite{HHJK2}, beginning with the following lemma.

\begin{lemma}
\label{fPDE.lem}For each fixed $\alpha,$ we let $\,f_{t}(\alpha,z)$ be the
function on the right-hand side of (\ref{uth}), but
without the supremum, namely%
\[
f_{t}(\alpha,z)=2g_{t}(\alpha)+(\alpha+t(a-b))\log\left\vert z\right\vert
^{2},
\]
or, explicitly,%
\begin{align*}
f_{t}(\alpha,z)  &  =g_{0}(\alpha)+\frac{2b}{a-b}\{[\alpha+t(a-b)]\log
[\alpha+t(a-b)]-\alpha\log\alpha\}-2bt\ \\
&  +(\alpha+t(a-b))\log\left\vert z\right\vert ^{2}.
\end{align*}
Then for each fixed $\alpha,$ the function $f_{t}(\alpha,z)$ satisfies the PDE
in (\ref{UtPDE}) as a function of $z.$
\end{lemma}

\begin{proof}
We compute that%
\begin{equation}
\frac{\partial f_{t}(\alpha,z)}{\partial t}=2b\log[\alpha+t(a-b)]+(a-b)\log
\left\vert z\right\vert ^{2} \label{dWdT}%
\end{equation}
and%
\[
z\frac{\partial f_{t}(\alpha,z)}{\partial z}=\bar{z}\frac{\partial
f_{t}(\alpha,z)}{\partial\bar{z}}=\alpha+t(a-b).
\]
Thus,%
\[
2b\log[\alpha+t(a-b)]=b\log\left(  z\frac{\partial f_{t}(\alpha,z)}{\partial
z}\right)  +b\log\left(  \bar{z}\frac{f_{t}(\alpha,z)}{\partial\bar{z}%
}\right)  .
\]
Plugging this result back into (\ref{dWdT}) and simplifying gives the claimed result.
\end{proof}

We now prove Theorem \ref{thePDEforS.thm} under Assumption
\ref{pde.assumption}.

\begin{proof}
[Proof of Theorem \ref{thePDEforS.thm}]Under the given assumptions, the
function%
\[
(z,t)\mapsto\alpha_{t}(|z|)=\mu_{t}(D_{|z|})
\]
will be smooth, where $D_{r}$ is closed disk of radius $r$ centered at $0$.
For $0<\left\vert z\right\vert <r_{\mathrm{out}}(t),$ the supremum (actually a
maximum) in \eqref{uth} is achieved at a unique value of $\alpha\in(0,1),$
equal to $\alpha_{t}(\left\vert z\right\vert )$, see, for instance
\cite[\S 3.2]{HHJK2}. Thus,
\begin{equation}
S_{t}(z)=f_{t}(\alpha_{t}(z),z). \label{StAnswer}%
\end{equation}
Then,%

\begin{align}
\frac{\partial S_{t}}{\partial t}(z)  &  =\frac{\partial f_{t}}{\partial
\alpha}(\alpha_{t}(z),z)\frac{\partial\alpha_{t}}{\partial t}(z)+\frac
{\partial f_{t} }{\partial t}(\alpha_{t}(z),z)\nonumber\\
&  =\frac{\partial f_{t}}{\partial t}(\alpha_{t}(z),z), \label{StDiff}%
\end{align}
and
\begin{align}
\frac{\partial S_{t}}{\partial z}(z)  &  =\frac{\partial f_{t}}{\partial
\alpha}(\alpha_{t}(z),z)\frac{\partial\alpha_{t}}{\partial z}(z)+\frac
{\partial f_{t}}{\partial z}(\alpha_{t}(z),z)\nonumber\\
&  =\frac{\partial f_{t}}{\partial z}(\alpha_{t}(z),z) \label{StDiffz}%
\end{align}
because $\partial f_{t}/\partial\alpha$ vanishes at $(\alpha_{t}(z),z),$ since
this point is the maximum over $\alpha.$ The result for $0<\left\vert
z\right\vert <r_{\mathrm{out}}(t)$ then follows from Lemma \ref{fPDE.lem}.

Meanwhile, for all $\left\vert z\right\vert >r_{\mathrm{out}}(t),$ the
supremum is achieved at $\alpha_{t}(z)\equiv1$ in a neighborhood of $(z,t)$.
In that case, \eqref{StDiff} and \eqref{StDiffz} still hold, but for a
different reason: $\partial\alpha_{t}/\partial t$ and $\partial\alpha
_{t}/\partial z$ vanish.

We now note that from (\ref{StAnswer}), we can see that $S_{t}(\alpha),$
$\partial S_{t}/\partial t,$ and $\partial S_{t}/\partial z$ are continuous
across the circle of radius $r_{\mathrm{out}}(t).$ Therefore, the domains can
be glued together and the PDE can be continuously extended to $\{(z,t):z\in
\mathbb{C }\backslash\{0\} , 0\le t <t_{\max}\}$ as claimed.
\end{proof}

\subsection{The Hamilton--Jacobi analysis\label{hj.sec}}

A first-order PDE on a domain $U$ in $\mathbb{R}^{n}$ is said to be of
Hamilton--Jacobi type if it has the form
\[
\frac{\partial}{\partial t}u(\mathbf{x},t)=-H\left(  x_{1},\ldots,x_{n}%
,\frac{\partial u}{\partial x_{1}},\ldots,\frac{\partial u}{\partial x_{n}%
}\right)
\]
for some \textquotedblleft Hamiltonian\textquotedblright\ function
$H(\mathbf{x},\mathbf{p})$ on $U\times\mathbb{R}^{n}\subset\mathbb{R}^{2n}.$
We then consider Hamilton's equations with Hamiltonian $H,$ that is, the
equations%
\begin{equation}
\frac{dx_{j}}{dt}=\frac{\partial H}{\partial p_{j}}(\mathbf{x}(t),\mathbf{p}%
(t));\quad\frac{dp_{j}}{dt}=-\frac{\partial H}{\partial x_{j}}(\mathbf{x}%
(t),\mathbf{p}(t)). \label{HamEq}%
\end{equation}
The \textbf{characteristic curves} of the problem are then solutions to
Hamilton's equations with arbitrary initial position $\mathbf{x}^{0}$ and
initial momentum $\mathbf{p}^{0}$ chosen as the derivatives of the initial
condition evaluated at $\mathbf{x}^{0}$:%
\begin{equation}
p_{j}^{0}=\frac{\partial}{\partial x_{j}^{0}}u(\mathbf{x}_{0},0).
\label{HamInitial}%
\end{equation}

The point of this construction is that we obtain nice formulas for $u$ and
$\nabla_{\mathbf{x}}u$ along the characteristic curves, as follows. If
$(\mathbf{x}(t),\mathbf{p}(t))$ solves (\ref{HamEq}) with initial condition as
in (\ref{HamInitial}), then we have the first Hamilton--Jacobi formula
\begin{equation}
u(\mathbf{x}(t),t)=u(\mathbf{x}_{0},0)-H(\mathbf{x}_{0},\mathbf{p}_{0}%
)~t+\int_{0}^{t}\mathbf{p}(s)\cdot\frac{d\mathbf{x}}{ds}~ds \label{HJformula1}%
\end{equation}
and the second Hamilton--Jacobi formula%
\begin{equation}
(\nabla_{\mathbf{x}}u)(\mathbf{x}(t),t)=\mathbf{p}(t). \label{HJformula2}%
\end{equation}
See Section 3.3 of the book \cite{Evans} of Evans or the proof of Proposition
5.3 in \cite{DHK}.

The PDE in Theorem \ref{thePDEforS.thm} is of Hamilton--Jacobi type, where the
Hamiltonian $H(x,y,p_{x},p_{y})$ is read off from the right-hand side of
(\ref{UtPDE}) by replacing $\partial S_{t}/\partial x$ with $p_{x}$ and
$\partial S_{t}/\partial y$ with $p_{y},$ with an overall minus sign. We then
use the complex-variable notations%
\begin{equation}
z=x+iy;\quad p=\frac{1}{2}(p_{x}-ip_{y}), \label{ZandP}%
\end{equation}
so that $p$ corresponds to $\partial S_{t}/\partial z$ on the right-hand side
of (\ref{UtPDE}). Thus, we find that
\begin{equation}
H(z,p)=-\log(z^{a}p^{b})-\log(\bar{z}^{a}\bar{p}^b). \label{theHamiltonian}%
\end{equation}

\begin{remark}
The PDE in (\ref{UtPDE}) in Theorem \ref{thePDEforS.thm} is of
Hamilton--Jacobi type, for which the Hamiltonian (\ref{theHamiltonian}) has a
very special form, namely the real part of a holomorphic function in $z$ and
$p.$ For PDEs of this special type, one can prove a push-forward theorem by
following the proof of Theorem 8.2 in \cite[Section 8.4]{HHptrf}. The theorem
would say that the measure obtained by taking the Laplacian of $S$ pushes
forward along the characteristic curves of the equation. This line of
reasoning could give a different proof of the push-forward result for
$P_{t}^{N}$ in Corollary \ref{push.cor}. Since we already have a more direct
proof, we will not pursue this line of reasoning.
\end{remark}

\begin{proposition}
\label{charCurves.prop}In the case $a\neq b,$ we have%
\begin{equation}
z(t)=z_{0}\left(  1+\frac{t(a-b)}{z_{0}p_{0}}\right)  ^{\frac{b}{b-a}},
\label{zchar1}%
\end{equation}
while in the case $a=b,$ we have%
\begin{equation}
z(t)=z_{0}\exp\left\{  -\frac{bt}{z_{0}p_{0}}\right\}  . \label{zchar2}%
\end{equation}

\end{proposition}

We note that the right-hand sides of (\ref{zchar1}) and (\ref{zchar2}) match
the formulas in Idea \ref{abMotion.idea}, after identifying $m_{0}(z_{0})$
with $p_{0}.$ This identification is natural because, by (\ref{HamInitial}),
the initial momentum is simply the Cauchy transform of the initial
distribution of zeros.

\begin{proof}
We have%
\begin{align*}
\frac{dp}{dt}  &  =\frac{1}{2}\left(  \frac{dp_{x}}{dt}-i\frac{dp_{y}}%
{dt}\right) \\
&  =-\frac{1}{2}\left(  \frac{\partial H}{\partial x}-i\frac{\partial
H}{\partial y}\right) \\
&  =-\frac{\partial H}{\partial z},
\end{align*}
where $\partial H/\partial z$ is the Wirtinger derivative (or Cauchy--Riemann
operator). Thus,%
\begin{equation}
\frac{dp}{dt}=\frac{a}{z}. \label{pPrime}%
\end{equation}

Meanwhile, we have%
\begin{equation}
\frac{dz}{dt}=\frac{dx}{dt}+i\frac{dp}{dt}=\frac{\partial H}{\partial p_{x}%
}+i\frac{\partial H}{\partial p_{y}}. \label{dzHam}%
\end{equation}
The Wirtinger derivative with respect to the variable $p$ in (\ref{ZandP}) is
computed as%
\[
\frac{\partial}{\partial p}=\frac{1}{2}\left(  \frac{\partial}{\partial
\operatorname{Re}p}-i\frac{\partial}{\partial\operatorname{Im}p}\right)  ,
\]
which works out to%
\[
\frac{\partial}{\partial p}=\frac{\partial}{\partial p_{x}}+i\frac{\partial
}{\partial p_{y}}.
\]
(The reader may check, for example, that applying $\partial/\partial p$ to $p$
gives 1 and applying $\partial/\partial p$ to $\bar{p}$ gives 0.) Then
(\ref{dzHam}) becomes%
\begin{equation}
\frac{dz}{dt}=\frac{\partial H}{\partial p}=-\frac{b}{p}. \label{zPrime}%
\end{equation}

From (\ref{pPrime}) and (\ref{zPrime}), we find that
\[
\frac{d}{dt}(zp)=a-b
\]
so that
\[
z(t)p(t)=z_{0}p_{0}+t(a-b).
\]
Then%
\[
\frac{dz}{dt}=-\frac{bz}{pz}=-z(t)\frac{b}{z_{0}p_{0}+t(a-b)}.
\]
This equation is separable and we can integrate it to%
\begin{equation}
\log z(t)-\log z_{0}=-\frac{b}{a-b}(\log(z_{0}p_{0}+t(a-b))-\log(z_{0}p_{0})),
\label{logZt}%
\end{equation}
which simplifies to the claimed expression.
\end{proof}

We remark that in the case $a=0$ and $b=1,$ the quantity $p(t)$ is independent
of $t$ by (\ref{pPrime}), so that (\ref{zPrime}) becomes (with $b=1$)
\begin{equation}
\frac{dz}{dt}=-\frac{1}{p_{0}}=\frac{1}{\frac{\partial S}{\partial z}%
(z_{0},0)} \label{dzRepeatedDiff}%
\end{equation}
as in (\ref{dz2}).

\section{The case $b<0$\label{bLessZero.sec}}

If $b<0,$ Theorem \ref{QNprofile.thm} still holds except when $a=b$. But the
exponential profile $g_{t}$ may not be concave on $(\alpha_{\min}^{t},1].$
Indeed, if $g_{0}=0$ (the case of the Kac polynomials), $g_{t}$ will be
\textit{convex} whenever $b<0.$ Nevertheless, Example \ref{LOstable.example}
is still valid when $b<0,$ as long as $b/(b-a)>0.$ Thus, we have
concavity---and therefore the push-forward results in Theorem \ref{push.thm}
and Corollary \ref{push.cor}---when $g_{0}$ is the exponential profile of a
Littlewood--Offord polynomial with parameter $\beta=b/(b-a)>0,$ even if $b<0.$

We note that Theorem \ref{KZ.thm} holds when the exponential profile is not
concave, provided that the exponential profile $g_{t}$ is replaced by its
concave majorant $G_{t},$ that is, the smallest concave function that is
everywhere greater than or equal to $g_{t}.$

In this section, we focus on the case of repeated integration, that is, $a=0$
and $b=-1.$ Then the exponential profile for $Q_{t}^{N}$ (which is obtained
from $P_{t}^{N}$ by stripping out the zeros at the origin) is%
\[
g_{t}(\alpha)=g_{0}(\alpha)-\{[\alpha+t]\log[\alpha+t]-\alpha\log\alpha\}-t,
\]
where we repeat that $g_{t}$ may or may not be concave, depending on the
choice of $g_{0}.$

\subsection{Singular behavior: the Kac case}

We begin by considering the case of the Kac polynomials, which correspond to
$g_{0}\equiv0.$ We focus on the polynomial $Q_{t}^{N},$ which is obtained from
$P_{t}^{N}$ by stripping away the uninteresting zeros at the origin. The
exponential profile of $Q_{t}^{N}$ with $a=0$ and $b=-1$ with $g_{0}\equiv0$
is, by Theorem \ref{QNprofile.thm},%
\[
g_{t}(\alpha)=-(\alpha+t)\log(\alpha+t)+\alpha\log\alpha+t,
\]
which is concave on $(0,1).$ The graph of the concave majorant of this
function is a straight line with slope%
\[
M_{t}=g_{t}(1)-g_{t}(0)=t\log t-(1+t)\log(1+t).
\]
Thus, the limiting root distribution of $Q_{t}^{N}$ will be concentrated
entirely on the ring of radius%
\[
e^{-M_{t}}=\frac{(1+t)^{1+t}}{t^{t}}=(1+t)\left(  1+\frac{1}{t}\right)  ^{t}.
\]
When $t$ is large, $e^{-M_{t}}\approx(1+t)e.$

\subsection{Singular behavior: the Weyl polynomial case}

As our next example, we consider the Weyl polynomials, which correspond to%
\[
g_{0}(\alpha)=-\frac{1}{2}(\alpha\log\alpha-\alpha).
\]
Then we compute%
\begin{align*}
g_{t}(\alpha)  &  =-(\alpha+t)\log(\alpha+t)+\frac{1}{2}\alpha\log\alpha+t\\
g_{t}^{\prime\prime}(\alpha)  &  =\frac{1}{2\alpha}-\frac{1}{\alpha+t},
\end{align*}
so that $g_{t}^{\prime\prime}(\alpha)$ is positive for $\alpha<t$ and negative
for $\alpha>t.$ Then for $t<1,$ we have a mix of convex and concave behavior
as in Figure \ref{nonconcave.fig}, while for $t\geq1$, we have that $g_{t}$ is convex.%

\begin{figure}[ptb]%
\centering
\includegraphics[
height=2.3091in,
width=3.7775in
]%
{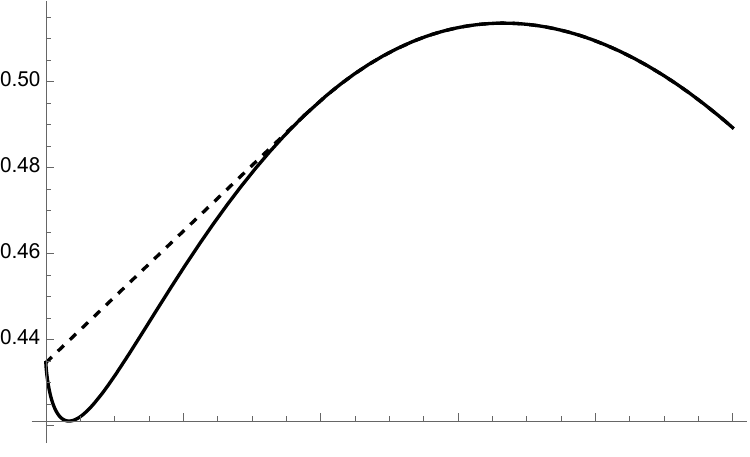}%
\caption{The exponential profile (solid) and its concave majorant (dashed) in
the Weyl case, for $t=0.15$}%
\label{nonconcave.fig}%
\end{figure}

Note that in the Weyl case, the function $\alpha_{0}(r)$ equals $r^{2},$ so
that with $a=0$ and $b=-1,$ the transport map in Definition \ref{Tt.def} takes
the form%
\[
T_{t}(re^{i\theta})=e^{i\theta}\left(  r+\frac{t}{r}\right)  .
\]
The singular behavior observed in this case can be attributed to the fact that
the magnitude of $T_{t}(re^{i\theta})$ is not an increasing function of $r.$
Thus, smaller roots can overtake larger roots and mass can accumulate on a circle.

We now compute the concave majorant $G_{t}$ of $g_{t}.$ If $t$ is small
enough, there will be a number $\alpha_{\mathrm{crit}}(t)\in(0,1)$ such that
$G_{t}$ is linear for $0\leq\alpha\leq\alpha_{\mathrm{crit}}(t)$ and will
agree with $g_{t}(\alpha)$ for $\alpha_{\mathrm{crit}}(t)<\alpha\leq1.$ (See,
again, Figure \ref{nonconcave.fig}.) For larger values of $t$, $G_{t}$ will be
linear on all of $[0,1].$ The precise range of $t$ for which a valid
$\alpha_{\mathrm{crit}}(t)$ exists will emerge from the calculation below. But
we note that if $t\geq1$ then $g_{t}$ is convex and thus the concave majorant
is certainly linear in this case.

The number $\alpha_{\mathrm{crit}}$ (if it exists) should be such that the
tangent line to the graph of $g_{t}$ at $\alpha_{\mathrm{crit}}$ hits the
$y$-axis at a height equal to $g_{t}(0)$:%
\[
\frac{g_{t}(\alpha_{\mathrm{crit}})-g_{t}(0)}{\alpha_{\mathrm{crit}}}%
=g_{t}^{\prime}(\alpha_{\mathrm{crit}})
\]
or, explicitly,%
\begin{equation}
\frac{\alpha_{\mathrm{crit}}+2t\log t-2t\log(\alpha_{\mathrm{crit}}%
+t)}{2\alpha_{\mathrm{crit}}}=0. \label{alphaCritCondition}%
\end{equation}
The condition (\ref{alphaCritCondition}) can be simplified to%
\[
\frac{\alpha_{\mathrm{crit}}}{t}=2\log\left(  1+\frac{\alpha_{\mathrm{crit}}%
}{t}\right)  ,
\]
so that
\[
\alpha_{\mathrm{crit}}(t)=tx,
\]
where $x$ is the unique positive solution to
\[
x=2\log(1+x),
\]
namely%
\[
x\approx2.513.
\]

If
\[
t<t_{\mathrm{crit}}:=\frac{1}{x}\approx0.3979,
\]
the value of $\alpha_{\mathrm{crit}}(t)$ will be less than 1. In this case, we
have may compute that%
\begin{align*}
e^{-G_{t}^{\prime}(\alpha)}  &  =e^{-G_{t}^{\prime}(\alpha_{\mathrm{crit}%
}(t))}\\
&  =\sqrt{t}\frac{1+x}{\sqrt{x}}\\
&  \approx2.216\sqrt{t}%
\end{align*}
for all $\alpha<\alpha_{\mathrm{crit}}(t).$ In this case, the limiting root
distribution $\sigma_{t}$ will have mass $\alpha_{\mathrm{crit}}=tx$ on the
circle of radius$\sqrt{t}(1+x)/\sqrt{x}$ while $\sigma_{t}$ will be absolutely
continuous outside this circle. See Figure \ref{intweyl.fig} in Section
\ref{flowFractional.sec}. (The figure shows the limiting root distribution of
$P_{t}^{N}$ rather than $Q_{t}^{N},$ so that there are roots at the origin.)

For $t\geq t_{\mathrm{crit}},$ no valid $\alpha_{\mathrm{crit}}(t)$ exists and
the concave majorant is simply linear. In this case, all of the mass of
$\sigma_{t}$ is concentrated on a single circle.

\subsection{Nonsingular behavior: the exponential polynomial
case\label{IntegrateExp.sec}}

We now take $g_{0}$ to be the exponential profile of the Littlewood--Offord
polynomials:
\[
g_{0}(\alpha)=-\beta(\alpha\log\alpha-\alpha).
\]
If $g_{t}$ is the exponential profile of $Q_{t}^{N}$ as in Theorem
\ref{QNprofile.thm} with $a=0$ and $b=-1$, we may easily compute that%
\[
g_{t}^{\prime\prime}(\alpha)=\frac{t(1-\beta)-\alpha\beta}{\alpha(t+\alpha)},
\]
with $\alpha_{\min}^{t}=0.$ When $\beta\geq1,$ we see that $g_{t}%
^{\prime\prime}(\alpha)\leq0$ for all $0\leq\alpha\leq1.$ Thus, we get a
concave exponential profile when performing repeated integration of
Littlewood--Offord polynomials with $\beta\geq1.$

We now focus our attention on the case of the exponential polynomials, that
is, the Littlewood--Offord polynomials with $\beta=1.$ The limiting root
distribution $\mu_{0}$ of the exponential polynomials is easily computed using
Theorem \ref{KZ.thm} and $\mu_{0}$ assigns mass $r$ to the disk of radius $r,$
for all $0\leq r\leq1.$ In that case, Example \ref{LOstable.example} with
$a=0$ and $b=-1$ applies. The limiting root distribution $\mu_{t}$ of
$P_{t}^{N}$ will have mass $t$ at the origin. The rest of the mass of $\mu
_{t}$ will be in an annulus having inner and outer radii given by%
\[
r_{\mathrm{in}}(t)=t;\quad r_{\mathrm{out}}(t)=1+t.
\]
In this annulus, $\mu_{t}$ will agree with the limiting root distribution of
the original exponential polynomial, dilated by a factor of $1+t.$ Explicitly,
$\mu_{t}$ has mass~$t/(1+t)$ at the origin and then $\mu_{t}$ assigns mass
$r/(1+t)$ to the disk of radius $r,$ for all $t\leq r\leq1+t.$ See Figure
\ref{integratedexp.fig}.%

\begin{figure}[ptb]%
\centering
\includegraphics[
height=2.5218in,
width=5.0687in
]%
{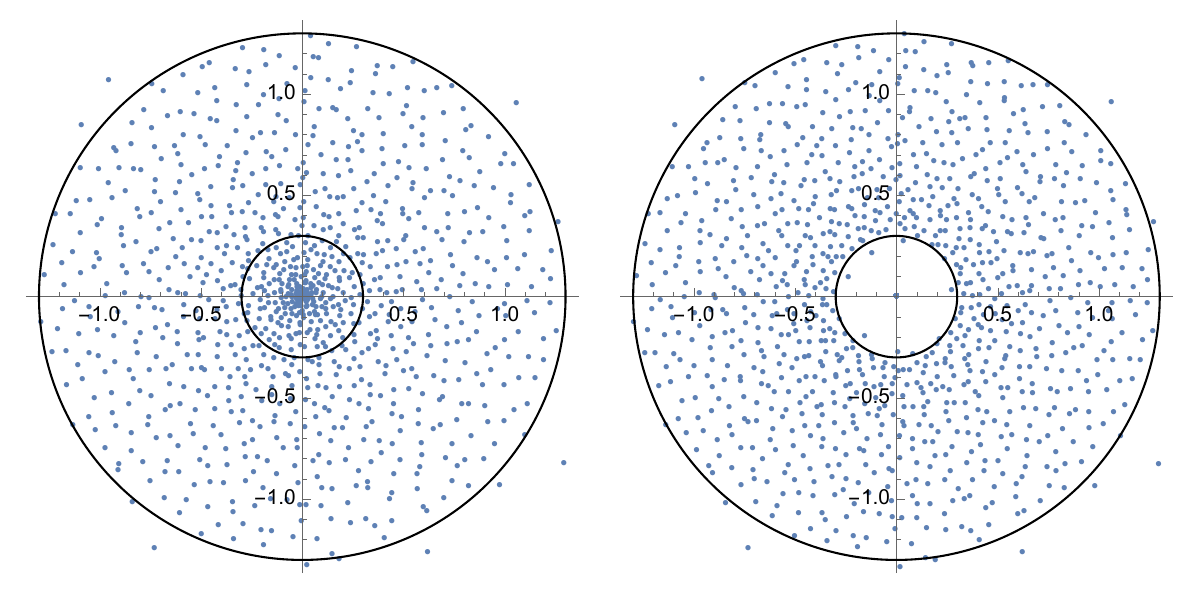}%
\caption{The roots of the original exponential polynomial, dilated by a factor
of $1+t$ (left), and the roots of $P_{t}^{N}$ (right). The inner circle has
radius $t.$ Shown for $t=0.3$ and $N=1,000.$}%
\label{integratedexp.fig}%
\end{figure}

The push-foward result in Theorem \ref{push.thm} applies in this case. Since
$\alpha_{0}(r)=r$ for the exponential polynomials, we get
\[
T_{t}(re^{i\theta})=e^{i\theta}\left(  r+t\right)  .
\]
We interpret the push-forward result to mean that each nonzero root moves
radially outward with constant speed equal to 1, which accounts for the
formulas for the inner and outer radii. Of course, roots of $P_{t}^{N}$ are
also being created at the origin. (By contrast, for repeated
\textit{differentiation} of the exponential polynomials, the roots move
radially \textit{inward} with speed 1, until they hit the origin.)

\appendix{}

\section{Formal derivation Idea \ref{radialMotion.idea}\label{PDEderiv.appendix}}

In this appendix, we give a heuristic argument for Idea \ref{radialMotion.idea}, namely that the roots move with constant velocity along curves $z(t)$ of the form \eqref{ztGen}. As discussed after the statement of Idea \ref{radialMotion.idea}, for this claim to be consistent with Idea \ref{singleDeriv.idea}, the Cauchy transform $m_t$, when evaluated along $z(t)$, must remain constant. 

The first step in the argument will be a heuristic derivation of a PDE for the log potential of the limiting root distribution. This PDE can then 
be analyzed by the method of characteristics and the characteristic curves turn out to be precisely the curves $z(t)$ in Idea \ref{radialMotion.idea}. 
Furthermore, the method of characteristics will tell us that the Cauchy transform (i.e., the $z$-derivative of the log potential) remains constant along these curves. At this point, it will be apparent that Idea \ref{radialMotion.idea} follows from Idea \ref{singleDeriv.idea}. 

The argument does not use the assumption that 
the initial root distribution is radial. Thus, at the heuristic level used in this appendix, it suffices to assume that we start with
a sequence of polynomials $P_0^N$ whose limiting root distribution is a compactly supported (but not necessarily radial) probability 
measure $\mu$. If $\mu_t$ is the limiting root distribution of $P_t^N$, then the argument below is valid except where the Cauchy transform 
of $\mu_t$ is zero. We then postulate that the roots die when they reach such a point. 

We define $P_{t}^{N}(z)$ as the $\lceil Nt\rceil$-th derivative of $P_{0}^{N},$
scaled by a convenient constant:%
\[
P_{t}^{N}(z)=\frac{1}{N^{Nt}}\left(  \frac{d}{dz}\right)  ^{\lceil Nt\rceil
}P_{0}^{N}(z).
\]
The constant is chosen so that the coefficients at time $t$ will have an asymptotic
behavior similar to that of the initial coefficients. 

\textbf{Motion of the zeros}. We expect (as in Idea \ref{singleDeriv.idea}) that a
single derivative will shift each root of $P_{t}^{N}$ by the negative
reciprocal of $m_{t},$ divided by the degree of $P_{t}^{N}$ --- which is
approximately $N(1-t).$ Now, applying a single derivative amounts to a change
in the time variable of $\Delta t=1/N.$ Thus, the zeros of $P_{t}^{N}$ should
be evolving approximately along curves $z(t)$ satisfying%
\[
\frac{\Delta z}{\Delta t}\approx\frac{-\frac{1}{N(1-t)m_{t}(z)}}{\frac{1}{N}}%
\]
or%
\begin{equation}
\frac{dz}{dt}\approx-\frac{1}{(1-t)m_{t}(z(t))}. \label{zeroMotion}%
\end{equation}

\textbf{The PDE\ for the log potential}. We define the log
potential $S^{N}(z,t)$ of $P_{t}^{N}(z)$ as%
\begin{equation}
S^{N}(z,t)=\frac{1}{N}\log\left(\left\vert P_{t}^{N}(z)\right\vert ^{2}\right)
.\label{logPotFromP}
\end{equation}
Here we intentionally divide by $N,$ the degree of the \textit{original} polynomial,
rather than by the degree of the polynomial at time $t.$ If the coefficient of
the highest-degree term in $P_{t}^{N}(z)$ is $a,$ then we have%
\begin{equation}
S^{N}(z,t)=\frac{1}{N}\sum_{j=1}^{(1-t)N}\log\left(\left\vert z-z_{j}(t)\right\vert^2\right)
+\frac{1}{N}\log\left\vert a\right\vert ^{2},\label{SNnorm}%
\end{equation}
where $\{z_{j}(t)\}_{j=1}^{(1-t)N}$ are the roots of $P_{t}^{N}(z).$ 

If $\mu_{t}$ is the limiting root distribution of $P_{t}^{N}(z),$ we expect to
recover $\mu_{t}$ from the large-$N$ limit $S$ of $S^{N}$ by an application of
the Laplace operator $\Delta$, with an extra factor of $1/(1-t)$ to account for the ``incorrect'' scaling in \eqref{logPotFromP}:
\[
\mu_{t}=\frac{1}{1-t}\frac{1}{4\pi}\Delta S^{N}(z,t).
\]
Thus, the limiting Cauchy transform will be%
\begin{equation}
m_{t}=\frac{1}{1-t}\frac{\partial S}{\partial z}.\label{mtFromS}%
\end{equation}

We now argue heuristically for the the following PDE for the large-$N$
limit $S$ of $S^{N}$:%
\begin{equation}
\frac{\partial S}{\partial t}=\log\left(  \left\vert \frac{\partial
S}{\partial z}\right\vert ^{2}\right)  , \label{diffPDE}%
\end{equation}
away from the origin. In the case of polynomials with real roots, a related
PDE (for the Cauchy transform of the limiting root distribution) was obtained
by Shlyakhtenko and Tao \cite[Eq. (1.18)]{STao}. We will verify the PDE
(\ref{diffPDE}) rigorously, but in a less direct way, in Section \ref{PDE.sec}.

We start by computing the (Wirtinger) derivative $\partial/\partial z$ of $S^N$ with respect to $z$ from \eqref{logPotFromP}. Since $\partial/\partial z$ treats $\bar P^N_t$ as a constant, we obtain:
\[
\frac{\partial S^{N}}{\partial z}=\frac{1}{N}\frac{\partial P_{t}^{N}/\partial z}{P_{t}^{N}(z)}
=\frac{1}{N}\frac{D^{\lceil Nt\rceil+1}P_{0}^{N}(z)}{D^{\lceil Nt\rceil}P_{0}^{N}(z)}.
\]
Then to approximate the $t$-derivative, we use a time-interval of $1/N,$ which
amounts to taking one additional derivative:
\begin{align*}
\frac{\partial S^{N}}{\partial t} & \approx\frac{1}{1/N}\frac{1}{N}\left(  \log\left(  \left\vert \frac{1}{N^{Nt+1}
}D^{\lceil Nt\rceil+1}P_{0}^{N}(z)\right\vert ^{2}\right)  -\log\left(
\left\vert \frac{1}{N^{Nt}}D^{\lceil Nt\rceil}P_{0}^{N}(z)\right\vert
^{2}\right)  \right)  \\
& =-2\log N+\log\left(  \left\vert \frac{D^{\lceil Nt\rceil+1}P_{0}^{N}%
(z)}{D^{\lceil Nt\rceil}P_{0}^{N}(z)}\right\vert ^{2}\right)  \\
& =\log\left(  \left\vert \frac{\partial S^{N}}{\partial z} \right\vert ^{2}\right),
\end{align*}
as claimed.

\textbf{Motion along the characteristic curves}.
Now, the PDE (\ref{diffPDE}) can be solved by the Hamilton--Jacobi method,
which is a form of the method of characteristics, as follows. (Details are
given in Section \ref{hj.sec}.) In the case at hand, the characteristic curves
$z_{\mathrm{char}}(t)$ are the solutions to
\begin{equation}
\frac{dz_{\mathrm{char}}}{dt}=-\frac{1}{\frac{\partial S}{\partial z}%
(z_{0},0)}, \label{dz2}%
\end{equation}
namely,%
\begin{equation}
z_{\mathrm{char}}(t)=z_{0}-\frac{t}{\frac{\partial S}{\partial z}(z_{0},0)}.
\label{zOfT}%
\end{equation}
Since (\ref{diffPDE}) is a constant-coefficient equation, the second
Hamilton--Jacobi equation says that $\partial S/\partial z$ is constant along
these curves:%
\begin{equation}
\frac{\partial S}{\partial z}(z_{\mathrm{char}}(t),t)=\frac{\partial
S}{\partial z}(z_{0},0). \label{mEqualsM}%
\end{equation}
Thus, (\ref{dz2}) can be rewritten as%
\begin{equation}
\frac{dz_{\mathrm{char}}}{dt}=-\frac{1}{\frac{\partial S}{\partial
z}(z_{\mathrm{char}}(t),t)}. \label{dz3}%
\end{equation}
But (\ref{dz3}) is, in light of (\ref{mtFromS}), precisely the equation we
proposed for the evolution of the zeros of the polynomial in (\ref{zeroMotion}%
). Thus, the zeros should move along the characteristic curves --- which are
the straight-line curves in (\ref{zOfT}). In the radial case, these curves will move
radially inward with constant speed, as in Idea \ref{radialMotion.idea}.

\subsection*{Acknowledgments}

BH is supported in part by a grant from the Simons Foundation. CH is supported
in part by the MoST grant 111-2115-M-001-011-MY3. JJ and ZK are funded by the
Deutsche Forschungsgemeinschaft (DFG, German Research Foundation) under
Germany's Excellence Strategy EXC 2044 - 390685587, Mathematics M\"unster:
\emph{Dynamics-Geometry-Structure} and have been supported by the DFG priority
program SPP 2265 \emph{Random Geometric Systems}.

The authors are grateful for useful conversations with Hari Bercovici, Andrew Campbell, and
David Renfrew. We also thank the referee for a very careful reading of the paper and for making numerous
helpful suggestions.

\end{document}